\documentclass{article}

\usepackage[english]{babel}

\usepackage{amsmath}
\usepackage{graphicx}
\usepackage{amssymb}
\usepackage[colorlinks=true, allcolors=blue]{hyperref}
\usepackage[utf8]{inputenc}
\usepackage{amsmath, ifthen, amsfonts, amssymb, srcltx, amsopn, xcolor, enumerate, comment} 
\usepackage{pb-diagram, pb-xy}
\usepackage{overpic}
\usepackage{mathrsfs}
\usepackage{tikzsymbols}
\usepackage{tikz-cd}
\usepackage{blindtext}
\usepackage{titlesec}

\usepackage{amsthm} 
\newtheorem{theorem}{Theorem}

\newtheorem{remark}[theorem]{Remark}
\newtheorem{example}[theorem]{Example}

\usepackage{xcolor,etoolbox}
\AtBeginEnvironment{ex}{\begingroup\color{blue}}
\AtEndEnvironment{ex}{\endgroup}
\AtBeginEnvironment{challenge}{\begingroup\color{orange}}
\AtEndEnvironment{challenge}{\endgroup}

\title{Geometry of Loci in Real Inner Product Spaces}
\author{Luis Chiner Carrillo}

\begin{document}
\maketitle

\begin{abstract}

This paper generalizes the notion of geometric curves such as hyperbolas and ellipses to more general vector spaces with an associated inner product. This is done by generalizing the definition in terms of loci and foci of said curves in Euclidean geometry to a general vector space with a real inner product, through which a norm can be induced. Through this generalization and focusing on the curves that are obtained through linear combinations of norms, we explore some properties of said curves. Specifically, we explore the addition of vectors in the curve, and in what other curves this addition can be found in relation to the original curve. Lastly, we observe the effects of applying the isomorphism to the geometric curve in the vector space onto $\mathbb{R}^n$ and we compare geometric curves obtained with the same definition in different vector spaces with different norms.

\end{abstract}
\pagebreak

\tableofcontents

\pagebreak

\section{Introduction}

This investigation is centered around loci in  real inner product vector spaces, that is, vector spaces from which a norm can be induced. Specifically, we investigate the loci that can be obtained as a linear combination of norms with respect to certain points, called focus points, or foci. We do so by generalizing the notion of loci as understood in Euclidean geometry to general vector spaces by using the norm of vectors (which in $\mathbb{R}^n$ with the natural operations and the standard inner product yields the Euclidean geometry) rather than the distance between points. 

\emph{Locus points}, or \emph{loci}, are points whose location is determined by one or more conditions. In other words, the location of these points satisfy a property or some properties and commonly form a line, a segment, a curve, or a surface.

Locus points are generally defined in relation to a set of points, known as \emph{focus points}, and a property of these points with regard to the location of the locus points. The type of relation between focus points and locus points that this investigation is centered around is given by an equation $g(x)=c$ where $g \colon \mathbb{R}^n \to \mathbb{R}$ is a linear function and $c \in \mathbb{R}$. The inputs of $g $ are the norms from the $n$ foci to the loci. 

\begin{example}
    
Consider the definition of an ellipse and its function. An ellipse is defined to be the set of all points whose sum of distances to two foci is constant. In this case, since there are two focus points,

\begin{displaymath}
g\colon \mathbb{R}^2 \to \mathbb{R}.
\end{displaymath}

If we consider these focus points to have coordinates $(-c,0), (c,0)$  and we set the constant that the two distances from the focus points to the locus points must add up to \begin{math}2a\end{math}, then we have that:

\begin{displaymath}g(\sqrt{(x+c)^2 + y^2}, \sqrt{(x-c)^2+y^2})= \sqrt{(x+c)^2 + y^2} + \sqrt{(x-c)^2+y^2} = 2a.\end{displaymath}

In this case, the linear relation observed in $g$ is $g(x,y) = x+y$, since the locus points we are looking for must have a constant sum of distances (as per the definition). Note also that in this case we are working in the vector space $\mathbb{R}^2$ with the standard operations and standard inner product. This induces the Euclidean geometry we are used to, however, we will work on generalizing this in this paper.

We can look at the graph and see that we obtain what we desire, an ellipse with focal points with coordinates $(-c,0), (c,0) $.

\begin{figure}[h!]
\centering
\includegraphics[width=0.75\textwidth]{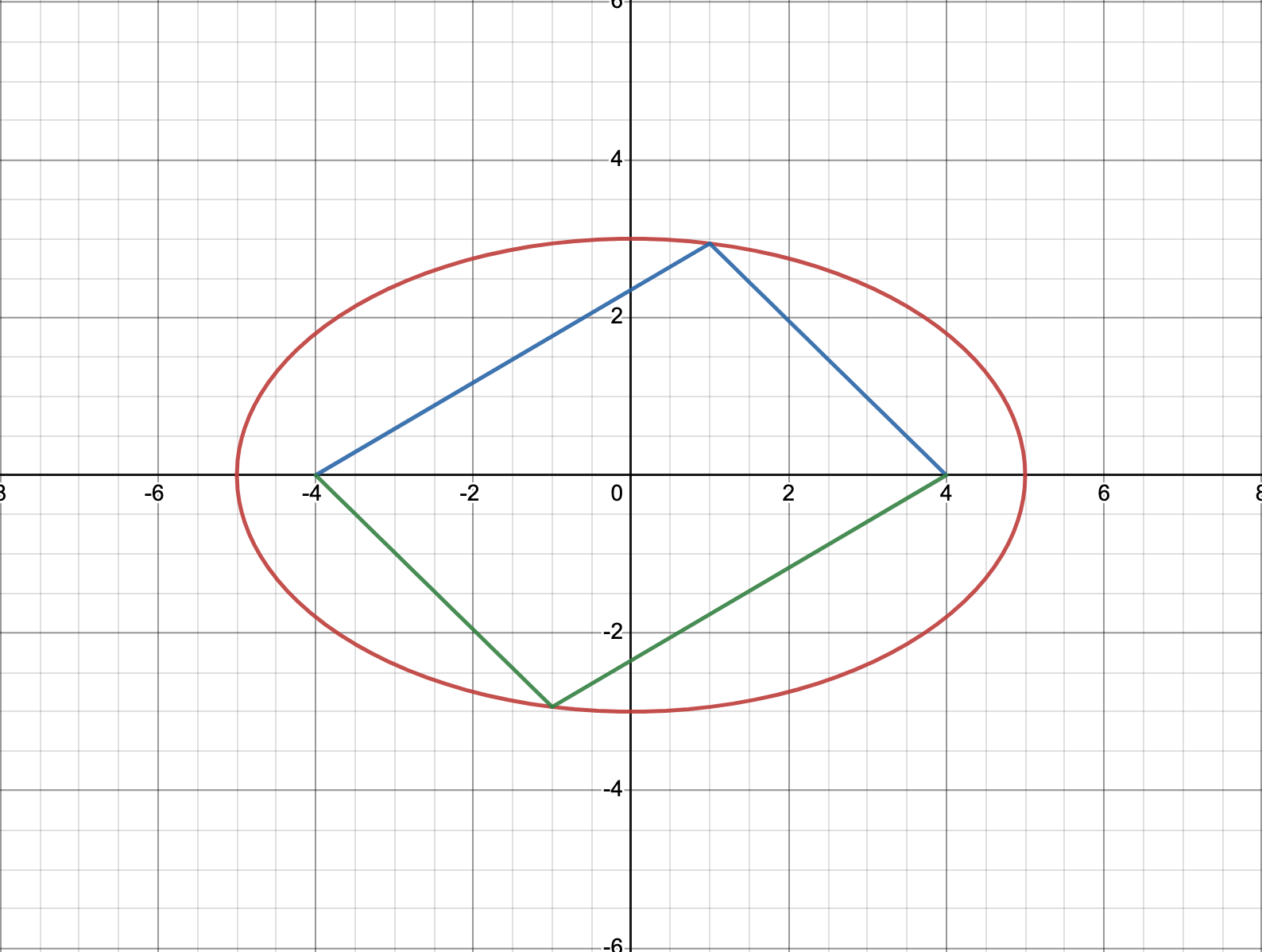}
\caption{\label{fig:graph}Graph of  $\sqrt{(x+c)^2 + y^2} + \sqrt{(x-c)^2+y^2} = 2a$ for $c = 4$, and $a = 5$}
\end{figure}

\end{example}

In this investigation we generalize loci (such as the ellipse), by understanding the norm of a vector as the equivalent to the length of a vector in $\mathbb{R}^2$

To begin with, we found an inequality relating vector addition between a vector in our loci and another vector in the vector space.

Given a function of the form:

\[
g(\left\| z - x_{1} \right \|, \dots, \left \| z-x_{n}\right\|) = \alpha_{1} \cdot \left \| z -x_{1}\right \| + \dots + \alpha_{n} \cdot \left \| z-x_{n} \right \| = c
\]

\begin{theorem}
Let $z \in L_{c}(x_{1}, \dots, x_{n})$ and $y \in V$. Then,

\begin{itemize}

\item[(1)] If

\[
c + (\sum_{i=1}^{n}\alpha_{i})\left \| y \right \| - \sum \alpha_{j} (\frac{\left \| z-x_{j} \right \|^{2} + \left \langle z-x_{j}, y \right \rangle}{\left \| z-x_{j} \right \|})  
\] 
\[
- \sum \alpha_{k} (\sqrt{\left \| y \right \|^{2} + \frac{(\left \| z-x_{k} \right \|^{2} + \left \langle z-x_{k}, y \right \rangle)^{2}}{\left \| z-x_{k} \right \|^{2}}}) < 0
\]
then:

\[
g(\left \| z+y-x_{1} \right \|, \dots, \left \| z+y-x_{n} \right \|) \geq 
c + g(\left \| y \right \|, \dots, \left \| y \right \|). 
\]

\item[(2)] If 

\[
c + (\sum_{i=1}^{n}\alpha_{i})\left \| y \right \| - \sum \alpha_{j} (\sqrt{\left \| y \right \|^{2} + \frac{(\left \| z-x_{k} \right \|^{2} + \left \langle z-x_{k}, y \right \rangle)^{2}}{\left \| z-x_{k} \right \|^{2}}})  
\] 
\[
- \sum \alpha_{k} (\frac{\left \| z-x_{j} \right \|^{2} + \left \langle z-x_{j}, y \right \rangle}{\left \| z-x_{j} \right \|})  > 0
\]
then:

\[
g(\left \| z+y-x_{1} \right \|, \dots, \left \| z+y-x_{n} \right \|) \leq 
c + g(\left \| y \right \|, \dots, \left \| y \right \|). 
\]

\item[(3)] If 

\[
(\sum_{i=1}^{n}\alpha_{i})\left \| y \right \| - c + \sum \alpha_{j} (\frac{\left \| z-x_{j} \right \|^{2} + \left \langle z-x_{j}, y \right \rangle}{\left \| z-x_{j} \right \|})  
\] 
\[
+ \sum \alpha_{k} (\sqrt{\left \| y \right \|^{2} + \frac{(\left \| z-x_{k} \right \|^{2} + \left \langle z-x_{k}, y \right \rangle)^{2}}{\left \| z-x_{k} \right \|^{2}}})  > 0
\]
then:

\[
g(\left \| z+y-x_{1} \right \|, \dots, \left \| z+y-x_{n} \right \|) \geq 
c - g(\left \| y \right \|, \dots, \left \| y \right \|). 
\]

\item[(4)] If \[
(\sum_{i=1}^{n}\alpha_{i})\left \| y \right \| - c + \sum \alpha_{j} (\sqrt{\left \| y \right \|^{2} + \frac{(\left \| z-x_{k} \right \|^{2} + \left \langle z-x_{k}, y \right \rangle)^{2}}{\left \| z-x_{k} \right \|^{2}}})  
\] 
\[
+ \sum \alpha_{k} (\frac{\left \| z-x_{j} \right \|^{2} + \left \langle z-x_{j}, y \right \rangle}{\left \| z-x_{j} \right \|})  < 0
\]
then:
\[
g(\left \| z+y-x_{1} \right \|, \dots, \left \| z+y-x_{n} \right \|) \leq 
c - g(\left \| y \right \|, \dots, \left \| y \right \|). 
\]

For all the statements above, all $\alpha_{j} > 0$ and all $\alpha_{k} < 0$

\end{itemize}

\end{theorem}

This theorem addresses the effects of adding a vector in the vector space to a vector in the loci, and how the resulting vector relates to the original loci. To see the proof for this theorem and the motivation behind the statement, please reference Theorem \ref{thm:theorem7}.

Similar to the statement of this theorem, we also were able to point down even further the properties of addition of two vectors in the loci and a linear combination of two vectors in the loci, with some restrictions on the coefficients that can be used for said linear combination.

This first theorem addresses the addition: 

\begin{theorem}
For any $v, w \in L_{c}(x_{1},\dots,x_{n})$, then

\begin{itemize}

\item[(1)] If

\[
2c - \sum \alpha_{j} (\frac{\left \| v-x_{j} \right \|^{2} + \left \langle z-x_{j}, w-x_{j} \right \rangle}{\left \| v-x_{j} \right \|})  
\] 
\[
- \sum \alpha_{k} (\sqrt{\left \| w-x_{k} \right \|^{2} + \frac{(\left \| v-x_{k} \right \|^{2} + \left \langle v-x_{k}, w-x_{k} \right \rangle)^{2}}{\left \| v-x_{k} \right \|^{2}}}) < 0, 
\] where 
 all $\alpha_{j} > 0$ and all $\alpha_{k} < 0$,
then:

\[
g(\left \| z+w-2x_{1} \right \|, \dots, \left \| v+w-2x_{n} \right \|) \geq 
2c. 
\]

\item[(2)] If 

\[
2c - \sum \alpha_{j} (\sqrt{\left \| w-x_{j} \right \|^{2} + \frac{(\left \| v-x_{j} \right \|^{2} + \left \langle v-x_{j}, w-x_{j} \right \rangle)^{2}}{\left \| v-x_{j} \right \|^{2}}})  
\] 
\[
- \sum \alpha_{k} (\frac{\left \| v-x_{j} \right \|^{2} + \left \langle v-x_{j}, w-x_{j} \right \rangle}{\left \| v-x_{j} \right \|})  > 0,
\] where 
 all $\alpha_{j} > 0$ and all $\alpha_{k} < 0$,
then:

\[
g(\left \| z+w-2x_{1} \right \|, \dots, \left \| v+w-2x_{n} \right \|) \leq 
2c.
\]

\item[(3)] If 

\[
 \sum \alpha_{j} (\frac{\left \| v-x_{j} \right \|^{2} + \left \langle v-x_{j}, w-x_{j} \right \rangle}{\left \| v-x_{j} \right \|})  
\] 
\[
+ \sum \alpha_{k} (\sqrt{\left \| w-x_{k} \right \|^{2} + \frac{(\left \| v-x_{k} \right \|^{2} + \left \langle v-x_{k}, w-x_{k} \right \rangle)^{2}}{\left \| v-x_{k} \right \|^{2}}})  > 0, 
\] where 
 all $\alpha_{j} > 0$ and all $\alpha_{k} < 0$,
then:

\[
g(\left \| z+w-2x_{1} \right \|, \dots, \left \| v+w-2x_{n} \right \|) \geq 
0.
\]

\item[(4)] If \[
\sum \alpha_{j} (\sqrt{\left \| w-x_{j} \right \|^{2} + \frac{(\left \| v-x_{j} \right \|^{2} + \left \langle v-x_{j}, w-x_{j} \right \rangle)^{2}}{\left \| v-x_{j} \right \|^{2}}})  
\] 
\[
+ \sum \alpha_{k} (\frac{\left \| v-x_{k} \right \|^{2} + \left \langle v-x_{k}, w-x_{k} \right \rangle}{\left \| v-x_{k} \right \|})  < 0,
\]where 
 all $\alpha_{j} > 0$ and all $\alpha_{k} < 0$,
 then:

\[
g(\left \| z+w-2x_{1} \right \|, \dots, \left \| v+w-2x_{n} \right \|) \geq 
0. 
\]

\end{itemize}

\end{theorem}

This second theorem references linear combinations of two vectors:

\begin{theorem}

Let $v,w \in L_{c}(x_{1},\dots,x_{n})$ and $\beta, \gamma \in \mathbb{R}_{\geq 0}$. Then:

\begin{itemize}

\item[(1)] If

\[
(\gamma+\beta)c - \sum \alpha_{j} (\frac{\left \| \gamma(v-x_{j}) \right \|^{2} + \left \langle \gamma(v-x_{j}), \beta(w-x_{j}) \right \rangle}{\left \| \gamma(v-x_{j}) \right \|})  
\] 
\[
- \sum \alpha_{k} (\sqrt{\left \| \beta(w-x_{k}) \right \|^{2} + \frac{(\left \| \gamma(v-x_{k}) \right \|^{2} + \left \langle \gamma(v-x_{k}), \beta(w-x_{j}) \right \rangle)^{2}}{\left \|\gamma(v-x_{k}) \right \|^{2}}}) < 0,\]
where all $\alpha_{j} > 0$ and all $\alpha_{k} < 0$, then:

\[
g(\left \| \gamma v+\beta w-(\gamma +\beta)x_{1} \right \|, \dots, \left \| \gamma v+\beta w-(\gamma + \beta)x_{n} \right \|) \geq 
(\gamma + \beta)c. 
\]

\item[(2)] If 

\[
(\gamma + \beta)c - \sum \alpha_{j} (\sqrt{\left \| \beta(w-x_{j}) \right \|^{2} + \frac{(\left \| \gamma(v-x_{j}) \right \|^{2} + \left \langle \gamma(v-x_{j}), \beta(w-x_{j}) \right \rangle)^{2}}{\left \| \gamma(v-x_{j}) \right \|^{2}}})  
\] 
\[
- \sum \alpha_{k} (\frac{\left \| \gamma(v-x_{k}) \right \|^{2} + \left \langle \gamma(v-x_{k}), \beta(w-x_{k}) \right \rangle}{\left \| \gamma(v-x_{k}) \right \|})  > 0, \] where all $\alpha_{j} > 0$ and all $\alpha_{k} < 0$, then:

\[
g(\left \| \gamma v+\beta w-(\gamma +\beta)x_{1} \right \|, \dots, \left \| \gamma v+\beta w-(\gamma + \beta)x_{n} \right \|) \leq 
(\gamma + \beta)c. 
\]

\item[(3)] If 

\[
(\beta-\gamma)c + \sum \alpha_{j} (\frac{\left \| \gamma(v-x_{j}) \right \|^{2} + \left \langle \gamma(v-x_{j}), \beta(w-x_{j}) \right \rangle}{\left \| \gamma(v-x_{j}) \right \|})  
\] 
\[
+ \sum \alpha_{k} (\sqrt{\left \| \beta(w-x_{k}) \right \|^{2} + \frac{(\left \| \gamma(v-x_{k}) \right \|^{2} + \left \langle \gamma(v-x_{k}), \beta(w-x_{j}) \right \rangle)^{2}}{\left \| \gamma(v-x_{k}) \right \|^{2}}})  > 0,\]
where all $\alpha_{j} > 0$ and all $\alpha_{k} < 0$, then:

\[
g(\left \| \gamma v+\beta w-(\gamma +\beta)x_{1} \right \|, \dots, \left \| \gamma v+\beta w-(\gamma + \beta)x_{n} \right \|) \geq 
(\gamma - \beta)c.  
\]

\item[(4)] If \[
(\beta-\gamma)c + \sum \alpha_{j} (\sqrt{\left \| \beta(w-x_{j}) \right \|^{2} + \frac{(\left \| \gamma(v-x_{j}) \right \|^{2} + \left \langle \gamma(v-x_{j}), \beta(w-x_{j}) \right \rangle)^{2}}{\left \| \gamma(v-x_{j}) \right \|^{2}}})  
\] 
\[
+ \sum \alpha_{k} (\frac{\left \| \gamma(v-x_{k}) \right \|^{2} + \left \langle \gamma(v-x_{k}), \beta(w-x_{k}) \right \rangle}{\left \| \gamma(v-x_{k}) \right \|})  < 0, \]where 
 all $\alpha_{j} > 0$ and all $\alpha_{k} < 0$, then

\[g(\left \| \gamma v+\beta w-(\gamma +\beta)x_{1} \right \|, \dots, \left \| \gamma v+\beta w-(\gamma + \beta)x_{n} \right \|) \leq 
(\gamma - \beta)c. \]

\end{itemize}

\end{theorem}

Reference Theorems \ref{thm:theorem8} and \ref{thm:theorem10}, respectively, for the proofs.

Having built this basis upon which we can better understand the properties of loci in vector spaces, we used the fact that any vector space is isomorphic to $\mathbb{R}^n$ where $n$ is the dimension of $V$, in order to find an isomorphism between the loci and $\mathbb{R}^n$

\begin{theorem}
    Let $V$ be an $n$-dimensional vector space over $\mathbb{R}$ with inner product $\left \langle \cdot, \cdot \right \rangle \to \mathbb{R}$ and ordered basis $\beta$. Applying the isomorphism $\phi_\beta \colon V \to \mathbb{R}^n$ to the locus points $L_{c}\subseteq V$ yields the same result as finding the locus points of the equation on $\mathbb{R}^n$ with the inner product defined via the isomorphism $\phi_{\beta}$.
\end{theorem}

The theorem above (see the proof in Theorem \ref{thm:theorem12}) states that applying the isomorphism to the loci in $V$ onto $\mathbb{R}^n$ yields the same result as finding the locus points of the equation on $\mathbb{R}^n$ with the inner product defined via the isomorphism.

\pagebreak

\section{Considerations and Definitions}
\label{sec:section2}

 \emph{An inner product on a vector space $ V  $ over $F$}  is a function that assigns to every ordered pair of vectors $ x, y \in V $, a scalar in $ F $, denoted as $ \left \langle x, y \right \rangle $, such that for all $ x, y, z \in V $ and all $ c \in F $ the following hold:

\begin{itemize}
\item $\left \langle x + z , y \right \rangle = \left \langle x, y \right \rangle + \left \langle z, y \right \rangle$
\item $\left \langle cx , y \right \rangle = c \cdot \left \langle x, y \right \rangle$
\item $\overline{\left \langle x + z , y \right \rangle} = \left \langle y, x \right \rangle$, where the $\bar{}$ denotes complex conjugation.
\item If $x \neq 0$, then $\left \langle x , x \right \rangle $  is a positive real number.
\end{itemize}

As such, locus or loci are studied within the field of geometry due to their undeniable applicability in this area. Examples in plane geometry of locus include amongst others:

\begin{itemize}
\item Set of points equidistant from two other points (Perpendicular Bisector)
\item Set of points at a constant distance from a fixed point (Circle)
\item Set of points for which the sum between two focal points is constant (Ellipse)
\end{itemize}

Though the concept may seem limited to geometry, loci can be easily applied to our real world, in particular in the study of movements of masses.

\emph{Real inner product spaces} are vector spaces with a defined operation, called inner product, which maps any ordered pair of vectors in the space to the field of real numbers $\mathbb{R}$.

The locus points that this investigation is centered around are the locus points that can be defined through a linear combination of the norms from the loci to a set of focus points. For the purposes of this investigation, we will understand generalized locus points in real inner product spaces as the set of all vectors $x \in V$ that satisfy the property that for a given function $g$ that is a linear combination of its variables and a constant $c$
\[
g(\left \|x-x_{1}\right \| ,\left \|x-x_{2}\right \|, \dots, \left \|x-x_{n}\right \|) =
\]
\[
\alpha_{1} \cdot \left \|x-x_{1}\right \| +  \dots + \alpha_{n} \cdot \left \|x-x_{n}\right \|
\], and

\[
g(\left \|x-x_{1}\right \| ,\left \|x-x_{2}\right \|, \dots, \left \|x-x_{n}\right \|) = c
\]

where $x_{1}, x_{2}, \dots, x_{n} \in V$ are the $n$ focus points of the locus we are interested in, and $\alpha_{1}, \dots, \alpha_{n} \in \mathbb{R}$.

At this stage, we define $L_{c}(x_{1}, \dots, x_{n})$ to be the following subset of $V$:

\begin{displaymath}
    L_{c}(x_{1}, \dots, x_{n}) = \left \{ x\in V \mid g(\left \|x-x_{1}\right \|, \dots, \left \|x-x_{n}\right \|) = c, x \in V, c \in \mathbb{R} \right \}
\end{displaymath}

In the vector space $\mathbb{R}^2$ with the Euclidean metric defined as:

\[
\left \langle (a_{1}, a_{2}),(b_{1},b_{2}) \right \rangle = a_{1} \cdot b_{1} + a_{2} \cdot b_{2} 
\]

Thus, we can define loci with our definition. Perhaps some well known loci in $\mathbb{R}^2$ are the ellipse and the hyperbola, defined as:

\[
g(\left \| x - x_{1} \right \|, \left \| x - x_{2} \right \|) =\left \| x - x_{1} \right \| + \left \| x - x_{2} \right \| = 2\cdot a
\]

\[
g(\left \| x - x_{1} \right \|, \left \| x - x_{2} \right \|) =\left \| x - x_{1} \right \| - \left \| x - x_{2} \right \| = 2\cdot a,
\]

respectively, where $x, x_{1}, x_{2} \in \mathbb{R}^2$ and $a\in \mathbb{R}$

\pagebreak
\subsection{Triangle Geometry in Vector Spaces}

Before we explore some properties of loci in the vector space and the subset $L_{c}(x_{1},\dots,x_{n}) \subseteq V$, we might find it helpful to develop some concepts of triangle geometry in general vector spaces with associated real inner products.

\begin{theorem}
\label{thm:theorem16}
   For any two vectors $x, y \in V$, where $V$ is a real inner product space, then the following is true:
   \[
   \left \| x + y \right \|^{2} = \left \| x \right \|^{2} + \left \| y \right \|^{2} + 2 \left \| x \right \| \cdot \left \| y \right \| \cdot \cos(\theta)
   \]
\end{theorem}

\begin{proof}
    \begin{figure}[h!]
\centering
\includegraphics[width=0.75\textwidth]{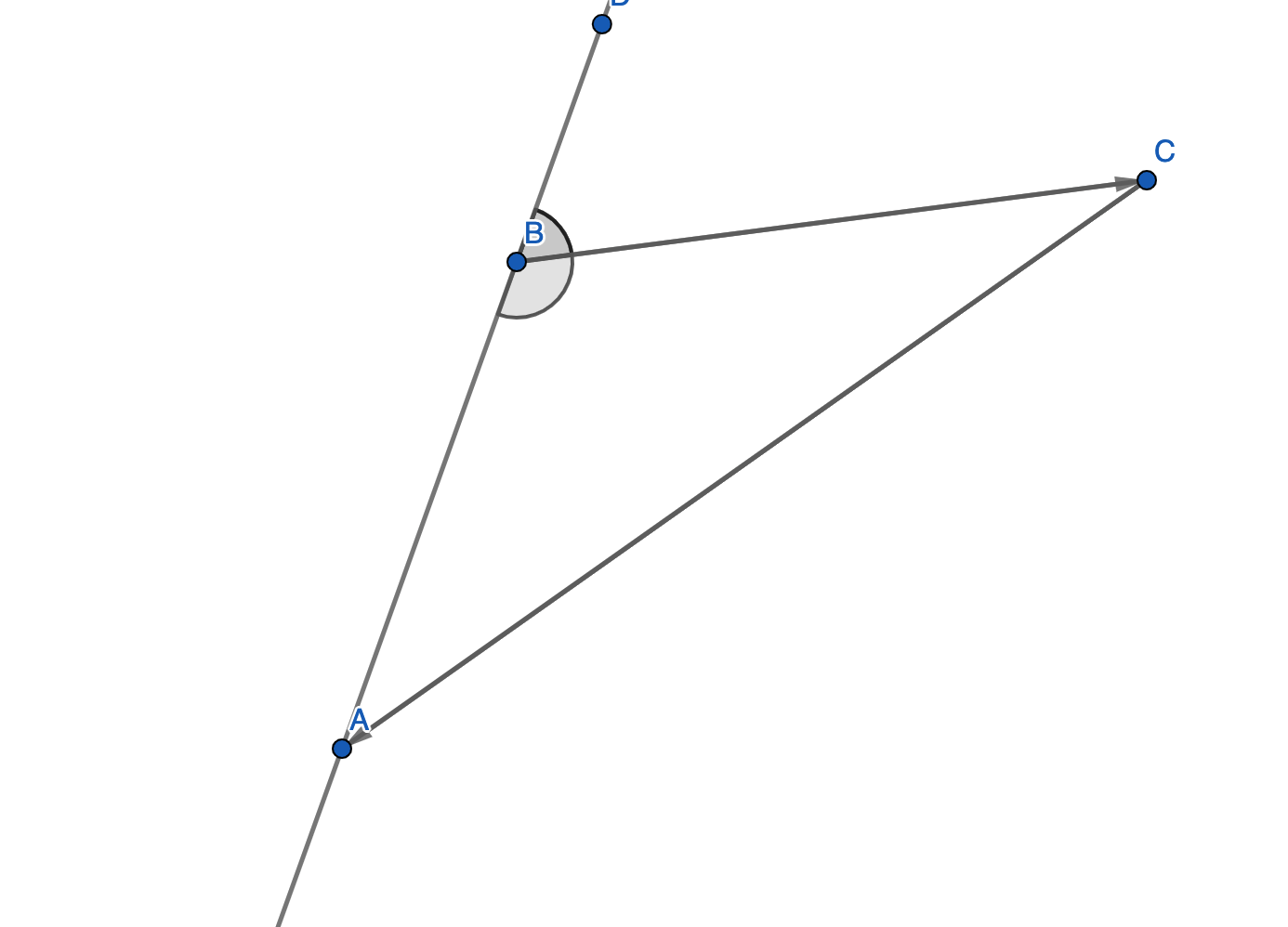}
\caption{\label{fig:graph43}Representation of a triangle with the vectors $x = \vec{AB}, y = \vec{BC}, x+y = \vec{AC} \in V$ }
\end{figure}

We will start off by realizing, as before, that:

\[
\left \| x + y \right \|^{2} = \left \langle x +y, x+y \right \rangle = \left \langle x, x \right \rangle + 2\left \langle x, y \right \rangle + \left \langle y,y \right \rangle =
\]
\[
\left \| x \right \|^{2} + \left \| y \right \|^{2} + 2\left \langle x, y\right \rangle
\]

Recall the definition of $\cos(\theta)$ in real inner product spaces:

\[
\cos(\theta) = \frac{\left \langle x, y \right \rangle}{\left \| x \right \| \cdot \left \| y \right \|}. 
\]

We realize, then, that

\[
\left \| x \right \| \cdot \left \| y \right \| \cdot \cos(\theta) = \left \langle x, y \right \rangle
\]

This can be substituted into our equation, and we obtain that:

\[
\left \| x + y \right \|^{2} =
\left \| x \right \|^{2} + \left \| y \right \|^{2} + 2\left \| x \right \| \cdot \left \| y \right \| \cdot \cos(\theta)
\]

\end{proof}

\begin{theorem}
    \label{thm:theorem17}
    For any real inner product vector space and a triangle with sides $x, y, x+y \in V$, the following holds true:

\[
\frac{\left \|x \right \|}{\sin(\arccos{\frac{\left \langle y, x+ y \right \rangle}{\left \|y \right \| \cdot \left \|x+y \right \|}})} = \frac{\left \|y \right \|}{\sin(\arccos{\frac{\left \langle x, x+ y \right \rangle}{\left \|x \right \| \cdot \left \|x+y \right \|}})} = \frac{\left \|x+y \right \|}{\sin(\arccos{\frac{\left \langle x, y \right \rangle}{\left \|x \right \| \cdot \left \|y \right \|}})} 
\]
    
\end{theorem}

\begin{proof}
    One must note that the angles $\arccos{\frac{\left \langle y, x+ y \right \rangle}{\left \|y \right \| \cdot \left \|x+y \right \|}} $, $\arccos{\frac{\left \langle x, x+ y \right \rangle}{\left \|x \right \| \cdot \left \|x+y \right \|}}$,
    $\arccos{\frac{\left \langle x, y \right \rangle}{\left \|x \right \| \cdot \left \|y \right \|}} $ are not the actual angles that are opposite to the sides $\left \|x \right \| $, $\left \|y \right \|$, and $\left \|x +y\right \|$, respectively, but rather their supplementary angles. However, $\sin(\theta) = \sin(\pi - \theta)$, so this does not affect our theorem.

    Firstly, we will calculate the values of $\sin(\theta)$ all three angles in the triangle. Of course, if $\cos(\theta) = \frac{\left \langle x, y \right \rangle}{\left \| x\right \| \cdot \left \| y\right \|}$, then we can apply the Pythagorean Theorem and obtain:

    \[
    \sin(\theta) = \sqrt{1-\frac{\left \langle x, y \right \rangle^{2}}{\left \| x\right \|^{2} \cdot \left \| y\right \|^{2}}}
    \]

    Knowing this, we will firstly prove the equality:

    \[
    \frac{\left \|y \right \|}{\sin(\arccos{\frac{\left \langle x, x+ y \right \rangle}{\left \|x \right \| \cdot \left \|x+y \right \|}})} = \frac{\left \|x+y \right \|}{\sin(\arccos{\frac{\left \langle x, y \right \rangle}{\left \|x \right \| \cdot \left \|y \right \|}})} 
    \]

We realize that:

\[
\frac{\left \|x+y \right \|}{\sin(\arccos{\frac{\left \langle x, y \right \rangle}{\left \|x \right \| \cdot \left \|y \right \|}})} = \frac{\left \|x+y \right \|}{\sqrt{1-\frac{\left \langle x, y \right \rangle^{2}}{\left \| x\right \|^{2} \cdot \left \| y\right \|^{2}}}} = 
\]
\[
\frac{\left \|x+y \right \|\left \|x \right \| \left \|y \right \|}{\sqrt{\left \|x \right \|^{2} \left \|y \right \|^{2}-\left \langle x, y \right \rangle ^{2}}} 
\]

We can simply the other term similarly, that is:

\[
 \frac{\left \|y \right \|}{\sin(\arccos{\frac{\left \langle x, x+ y \right \rangle}{\left \|x \right \| \cdot \left \|x+y \right \|}})} = \frac{\left \|y \right \|}{\sqrt{1-\frac{\left \langle x, x+y \right \rangle^{2}}{\left \| x\right \|^{2} \cdot \left \| x+y\right \|^{2}}}} = 
\]
\[
\frac{\left \|x+y \right \|\left \|x \right \| \left \|y \right \|}{\sqrt{\left \|x \right \|^{2} \left \|x+y \right \|^{2}-(\left \langle x, x \right \rangle + \left \langle x, y \right \rangle) ^{2}}} =
\]
\[
\frac{\left \|x+y \right \|\left \|x \right \| \left \|y \right \|}{\sqrt{\left \|x \right \|^{2} \left \|x+y \right \|^{2}-\left \| x \right \|^{2}\left \| x \right \|^{2} -2 \left \| x\right \|^{2}\left \langle x, y \right \rangle - \left \langle x, y \right \rangle^{2}}} = 
\]
\[
\frac{\left \|x+y \right \|\left \|x \right \| \left \|y \right \|}{\sqrt{\left \|x \right \|^{2} (\left \|x+y \right \|^{2}-\left \| x \right \|^{2} -2\left \langle x, y \right \rangle) - \left \langle x, y \right \rangle^{2}}}
\]

Using the Cosine Law and properties we derived above, we have that:

\[
\left \|x+y \right \|^{2}-\left \| x \right \|^{2} -2\left \langle x, y \right \rangle = \left \| y \right \|^{2}
\]

Hence, the equation above is reduced to:
\[
\frac{\left \|x+y \right \|\left \|x \right \| \left \|y \right \|}{\sqrt{\left \|x \right \|^{2} \left \|y \right \| - \left \langle x, y \right \rangle^{2}}}
\]

Since both terms are equal, it means that

\[
    \frac{\left \|y \right \|}{\sin(\arccos{\frac{\left \langle x, x+ y \right \rangle}{\left \|x \right \| \cdot \left \|x+y \right \|}})} = \frac{\left \|x+y \right \|}{\sin(\arccos{\frac{\left \langle x, y \right \rangle}{\left \|x \right \| \cdot \left \|y \right \|}})} 
\]

The proof for the third equality is similar to this one. Hence, we have that:

\[
\frac{\left \|x \right \|}{\sin(\arccos{\frac{\left \langle y, x+ y \right \rangle}{\left \|y \right \| \cdot \left \|x+y \right \|}})} = \frac{\left \|y \right \|}{\sin(\arccos{\frac{\left \langle x, x+ y \right \rangle}{\left \|x \right \| \cdot \left \|x+y \right \|}})} = \frac{\left \|x+y \right \|}{\sin(\arccos{\frac{\left \langle x, y \right \rangle}{\left \|x \right \| \cdot \left \|y \right \|}})} 
\]

\end{proof}

We have now proved the Cosine and Sine Laws for any general real inner product space. Lastly, we will prove one last theorem, which will come in handy in the next subsections. This theorem will help us calculate the length of the vector $x+y$ without the need of directly applying either the Cosine or Sine Law.

\begin{theorem}
\label{thm:theorem18}
    The length of the vector $x+y$ is given by:

    \[
    \left \| x + y \right \|^{2} = \left \| y \right \|^{2} \cdot \sin^{2}(\arccos \frac{\left \langle x, y \right \rangle}{ \left \| x \right \| \left \| y \right \|}) + \frac{(\left \|x \right \|^{2} + \left \langle x, y \right \rangle)^{2}}{\left \| x \right \|^{2}}
    \]
\end{theorem}

\begin{proof}
    We will firstly use the Sine Law that we just proved in order to realize that:

    \[
    \frac{\left \|y \right \|}{\sin(\arccos{\frac{\left \langle x, x+ y \right \rangle}{\left \|x \right \| \cdot \left \|x+y \right \|}})} = \frac{\left \|x+y \right \|}{\sin(\arccos{\frac{\left \langle x, y \right \rangle}{\left \|x \right \| \cdot \left \|y \right \|}})} 
    \]

    We can get rid of the fraction by multiplying both sides by the denominator of the other. In doing so, we obtain:
    \[
\left \|x+y \right \|\sin(\arccos{\frac{\left \langle x, x+ y \right \rangle}{\left \|x \right \| \cdot \left \|x+y \right \|}}) =     \left \|y \right \|\sin(\arccos{\frac{\left \langle x, y \right \rangle}{\left \|x \right \| \cdot \left \|y \right \|}})
    \]

We can square both sides to obtain:

    \[
\left \|x+y \right \|^{2}\sin^{2}(\arccos{\frac{\left \langle x, x+ y \right \rangle}{\left \|x \right \| \cdot \left \|x+y \right \|}}) =     \left \|y \right \|^{2}\sin^{2}(\arccos{\frac{\left \langle x, y \right \rangle}{\left \|x \right \| \cdot \left \|y \right \|}})
    \]

If we now add the term $ \left \|x+y \right \|^{2}\cos^{2}(\arccos{\frac{\left \langle x, x+ y \right \rangle}{\left \|x \right \| \cdot \left \|x+y \right \|}}) $ to both sides, then we obtain, by the Pythagorean Theorem:

    \[
\left \|x+y \right \|^{2} =     \left \|y \right \|^{2}\sin^{2}(\arccos{\frac{\left \langle x, y \right \rangle}{\left \|x \right \| \cdot \left \|y \right \|}}) + \left \|x+y \right \|^{2}\cos^{2}(\arccos{\frac{\left \langle x, x+ y \right \rangle}{\left \|x \right \| \cdot \left \|x+y \right \|}})
    \]

But, of course:

\[
\cos^{2}(\arccos{\frac{\left \langle x, x+ y \right \rangle}{\left \|x \right \| \cdot \left \|x+y \right \|}}) = (\cos(\arccos{\frac{\left \langle x, x+ y \right \rangle}{\left \|x \right \| \cdot \left \|x+y \right \|}}))^{2} =
\]
\[
(\frac{\left \langle x, x+ y \right \rangle}{\left \|x \right \| \cdot \left \|x+y \right \|})^{2} = (\frac{\left \|x \right \|^{2} + \left \langle x, y \right \rangle}{\left \|x \right \| \cdot \left \|x+y \right \|})^{2} = \frac{(\left \|x \right \|^{2} + \left \langle x, y \right \rangle)^{2}}{\left \|x \right \|^{2} \cdot \left \|x+y \right \|^{2}}
\]

This means that:

    \[
\left \|x+y \right \|^{2} =     \left \|y \right \|^{2}\sin^{2}(\arccos{\frac{\left \langle x, y \right \rangle}{\left \|x \right \| \cdot \left \|y \right \|}}) + \left \|x+y \right \|^{2}\cdot \frac{(\left \|x \right \|^{2} + \left \langle x, y \right \rangle)^{2}}{\left \|x \right \|^{2} \cdot \left \|x+y \right \|^{2}} = 
    \]

    \[
 \left \|y \right \|^{2}\sin^{2}(\arccos{\frac{\left \langle x, y \right \rangle}{\left \|x \right \| \cdot \left \|y \right \|}}) +  \frac{(\left \|x \right \|^{2} + \left \langle x, y \right \rangle)^{2}}{\left \|x \right \|^{2} }   
    \]

\end{proof}

\pagebreak

\section{General Vector Spaces}

Now that we have in mind a clear definition of what locus points over real inner product spaces are, we can now move on to investigating some of the properties that these type of locus points have.

\subsection{Subset \begin{math} L_{c}(x_{1},\dots,x_{n}) \subseteq V \end{math}}

Now that we have cleared some technicalities and understand the definitions, let us work in a general vector space $V$, rather than in the specific case there $V= \mathbb{R}^n$. 

Recall that

\[
    L_{c}(x_{1}, \dots, x_{n}) = \]
    \[\left \{ x\in V \mid g(\sqrt{\left \langle x-x_{1}, x-x_{1} \right \rangle}, \dots, \sqrt{\left \langle x-x_{n}, x-x_{n} \right \rangle}) = c, c \in \mathbb{R} \right \}.
\]

At this point, we might be rightfully wondering what properties do the vectors in $L_{c}(x_{1}, \dots, x_{n})$ have. For this endeavor, recall the general Cauchy-Schwarz inequality:

\[
\left \langle x, y \right \rangle \leq \sqrt{\left \langle x, x \right \rangle}\sqrt{\left \langle y, y \right \rangle}, \quad \text{where} \quad x,y\in V.
\]

Let $z \in L_{c}(x_{1}, \dots, x_{n})$. This means that

\[
g(\sqrt{\left \langle z-x_{1}, z-x_{1} \right \rangle}, \sqrt{\left \langle z-x_{2}, z-x_{2} \right \rangle}, \dots, \sqrt{\left \langle z-x_{n}, z-x_{n} \right \rangle}) = c
\]

Unfortunately, the subset $L_{c}(x_{1},\dots,x_{n}) \subseteq V$ is not a subspace. For instance,

\[
g(\sqrt{\left \langle 0-x_{1}, 0-x_{1} \right \rangle}, \sqrt{\left \langle 0-x_{2}, 0-x_{2} \right \rangle}, \dots, \sqrt{\left \langle 0-x_{n}, 0-x_{n} \right \rangle}) =
\]
\[
\alpha_{1}\left \|x_{1} \right \| + \dots + \alpha_{n}\left \|x_{n} \right \|.
\]

Then, it might be that $0_{V} \notin L_{c}(x_{1},\dots,x_{n})$. But even if $0_{V} \in L_{c}(x_{1},\dots,x_{n})$, if we have two vectors $v,w \in L_{c}(x_{1},\dots,x_{n})$, then $v+w$ is not necessarily in $ L_{c}(x_{1},\dots,x_{n})$.

It is at this point where our triangle geometry from \ref{sec:section2} will prove to very helpful. In order to formulate the next theorem, we will utilize the formula that we derived to calculate the norm of the addition of vectors in a given vector space, in terms of the norms and inner product of the individual vectors, and the angle in between them. This will help us create a lower and upper bound for the actual value of $g(\left \| z+y-x_{1} \right \|, \dots, \left \| z+y-x_{n} \right \|)$

\begin{theorem}
\label{thm:theorem7}
Let $z \in L_{c}(x_{1}, \dots, x_{n})$ and $y \in V$. Then,

\begin{itemize}

\item[(1)] If

\[
c + (\sum_{i=1}^{n}\alpha_{i})\left \| y \right \| - \sum \alpha_{j} (\frac{\left \| z-x_{j} \right \|^{2} + \left \langle z-x_{j}, y \right \rangle}{\left \| z-x_{j} \right \|})  
\] 
\[
- \sum \alpha_{k} (\sqrt{\left \| y \right \|^{2} + \frac{(\left \| z-x_{k} \right \|^{2} + \left \langle z-x_{k}, y \right \rangle)^{2}}{\left \| z-x_{k} \right \|^{2}}}) < 0
\]
then:

\[
g(\left \| z+y-x_{1} \right \|, \dots, \left \| z+y-x_{n} \right \|) \geq 
c + g(\left \| y \right \|, \dots, \left \| y \right \|). 
\]

\item[(2)] If 

\[
c + (\sum_{i=1}^{n}\alpha_{i})\left \| y \right \| - \sum \alpha_{j} (\sqrt{\left \| y \right \|^{2} + \frac{(\left \| z-x_{k} \right \|^{2} + \left \langle z-x_{k}, y \right \rangle)^{2}}{\left \| z-x_{k} \right \|^{2}}})  
\] 
\[
- \sum \alpha_{k} (\frac{\left \| z-x_{j} \right \|^{2} + \left \langle z-x_{j}, y \right \rangle}{\left \| z-x_{j} \right \|})  > 0
\]
then:

\[
g(\left \| z+y-x_{1} \right \|, \dots, \left \| z+y-x_{n} \right \|) \leq 
c + g(\left \| y \right \|, \dots, \left \| y \right \|). 
\]

\item[(3)] If 

\[
(\sum_{i=1}^{n}\alpha_{i})\left \| y \right \| - c + \sum \alpha_{j} (\frac{\left \| z-x_{j} \right \|^{2} + \left \langle z-x_{j}, y \right \rangle}{\left \| z-x_{j} \right \|})  
\] 
\[
+ \sum \alpha_{k} (\sqrt{\left \| y \right \|^{2} + \frac{(\left \| z-x_{k} \right \|^{2} + \left \langle z-x_{k}, y \right \rangle)^{2}}{\left \| z-x_{k} \right \|^{2}}})  > 0
\]
then:

\[
g(\left \| z+y-x_{1} \right \|, \dots, \left \| z+y-x_{n} \right \|) \geq 
c - g(\left \| y \right \|, \dots, \left \| y \right \|). 
\]

\item[(4)] If \[
(\sum_{i=1}^{n}\alpha_{i})\left \| y \right \| - c + \sum \alpha_{j} (\sqrt{\left \| y \right \|^{2} + \frac{(\left \| z-x_{k} \right \|^{2} + \left \langle z-x_{k}, y \right \rangle)^{2}}{\left \| z-x_{k} \right \|^{2}}})  
\] 
\[
+ \sum \alpha_{k} (\frac{\left \| z-x_{j} \right \|^{2} + \left \langle z-x_{j}, y \right \rangle}{\left \| z-x_{j} \right \|})  < 0
\]
then:
\[
g(\left \| z+y-x_{1} \right \|, \dots, \left \| z+y-x_{n} \right \|) \leq 
c - g(\left \| y \right \|, \dots, \left \| y \right \|). 
\]

For all the statements above, all $\alpha_{j} > 0$ and all $\alpha_{k} < 0$

\end{itemize}

\end{theorem}

\begin{proof}
(1) Note that

\[
\sqrt{\left \langle z+y-x_{i}, z+y-x_{i} \right \rangle} = \sqrt{\left \langle z-x_{i}+y, z-x_{i} +y \right \rangle} = 
\]
\[
\sqrt{\left \langle (z-x_{i})+y, (z-x_{i}) +y \right \rangle} = \sqrt{\left \langle z-x_{i}, z-x_{i} \right \rangle + 2\left \langle z-x_{i}, y \right \rangle + \left \langle y,y \right \rangle }.
\]

By the Cauchy-Schwarz inequality,

\[
2 | \left \langle z-x_{i}, y \right \rangle | \leq 2 \left \| z-x_{i}\right \| \left \| y\right \|.
\]

If $\left \langle z-x_{i}, y \right \rangle \geq 0$, this means that
 \[
 \left \langle z-x_{i}, z-x_{i} \right \rangle + 2\left \langle z-x_{i}, y \right \rangle + \left \langle y,y \right \rangle \leq \left \langle z-x_{i}, z-x_{i} \right \rangle + 2\left \| z-x_{i}\right \| \left \| y\right \| + \left \langle y,y \right \rangle.
 \]

 But we also know that
 
 \[
 \left \langle z-x_{i}, z-x_{i} \right \rangle + 2 \left \| z-x_{i}\right \| \left \| y\right \| +  \left \langle y,y \right \rangle = (\left \| z-x_{i}\right \| + \left \| y\right \|)^2.
 \]

 Hence,

\[
\left \| z+y-x_{i} \right \| \leq \left \| z-x_{i} \right \| + \left \| y \right \|.
\]

We can look at the bigger picture of $g$ and obtain that

\[
g(\left \| z+y-x_{1} \right \|, \dots, \left \| z+y-x_{n} \right \|) = 
\]
\[
\alpha_{1}\left \| z+y-x_{1} \right \| + \dots + \alpha_{n}\left \| z+y-x_{n} \right \|.
\]

Here we must be careful, because if $\alpha_{k} < 0$, then

\[\alpha_{k}\left \| z+y-x_{k} \right \| > \alpha_{k}(\left \| z-x_{1} \right \| + \left \| y \right \|).
\]

For that reason, we will have to find an estimate for the overall change in the function $g$ to determine how it is affected by the addition of the vector $y$. In order to do this, we will make use of Theorem \ref{thm:theorem18}. Using Theorem \ref{thm:theorem18} by letting $x = z-x_{k}$ and $y = y$, we have that:

\[
\left \| z-x_{k}+y \right \|^{2} = \left \| y\right \|^{2}\sin^{2}(\arccos(\frac{\left \langle x, y \right \rangle}{\left \| x \right \| \left \| y \right \|})) + \frac{(\left \|z-x_{k} \right \|^{2} + \left \langle z-x_{k}, y\right \rangle)^{2}}{\left \|z-x_{k} \right \|^{2}})
\]

We also know that the function $0 \leq \sin^{2}(x) \leq 1$, so we can easily find some bounds for the value of $\left \| z-x_{k}+y \right \|$:

\[
\frac{\left \|z-x_{k} \right \|^{2} + \left \langle z-x_{k}, y\right \rangle}{\left \|z-x_{k} \right \|} \leq \left \| z-x_{k}+y \right \| \leq
\]
\[
\sqrt{\left \| y\right \|^{2} + \frac{(\left \|z-x_{k} \right \|^{2} + \left \langle z-x_{k}, y\right \rangle)^{2}}{\left \|z-x_{k} \right \|^{2}}}
\]

So, in order to estimate the change in the function $g$ and get some bounds as to where the resulting value can be, one can realize that we must calculate:
 
\[
\sum_{i=1}^{n}\alpha_{i}(\left \| z-x_{i} \right \| + \left \| y \right \| - \left \| z+y-x_{i} \right \|)
\]

Of course, using the exact value of $\left \| z+y-x_{i} \right \|$ would be of no use since, knowing its value, one could easily determine the value of the function. For that reason, we will make use of the inequality observed above and note that. Due to space constraints, we will define $\delta_{i}$ to be the approximation of the term $\left \| z+y-x_{i} \right \|$:

\[
-(\frac{\left \|z-x_{k} \right \|^{2} + \left \langle z-x_{k}, y\right \rangle}{\left \|z-x_{k} \right \|}) \geq -\left \| z-x_{k}+y \right \| \geq
\]
\[
-\sqrt{\left \| y\right \|^{2} + \frac{(\left \|z-x_{k} \right \|^{2} + \left \langle z-x_{k}, y\right \rangle)^{2}}{\left \|z-x_{k} \right \|^{2}}}
\]

Note that we change the signs because the value we are trying to approximate has the term $-\left \| z-x_{k}+y \right \|$.

In this case, if we have that $\alpha_{k} > 0$, then the inequality is not affected and:

\[
-\alpha_{k}(\frac{\left \|z-x_{k} \right \|^{2} + \left \langle z-x_{k}, y\right \rangle}{\left \|z-x_{k} \right \|}) \geq -\alpha_{k}\left \| z-x_{k}+y \right \| \geq
\]
\[
-\alpha_{k}\sqrt{\left \| y\right \|^{2} + \frac{(\left \|z-x_{k} \right \|^{2} + \left \langle z-x_{k}, y\right \rangle)^{2}}{\left \|z-x_{k} \right \|^{2}}}
\]

However, if $\alpha_{k} < 0$, then:

\[
-\alpha_{k}(\frac{\left \|z-x_{k} \right \|^{2} + \left \langle z-x_{k}, y\right \rangle}{\left \|z-x_{k} \right \|}) \leq -\alpha_{k}\left \| z-x_{k}+y \right \| \leq
\]
\[
-\alpha_{k}\sqrt{\left \| y\right \|^{2} + \frac{(\left \|z-x_{k} \right \|^{2} + \left \langle z-x_{k}, y\right \rangle)^{2}}{\left \|z-x_{k} \right \|^{2}}}
\]

This hints to the fact that, depending on the value of $\alpha_{k}$ we must choose one value for the estimate or another. Recall that $\delta_{i}$ is our approximation of $\left \| z+y-x_{i} \right \|$ 

\[
\sum_{i=1}^{n}\alpha_{i}(\left \| z-x_{i} \right \| + \left \| y \right \| - \delta_{i}) = 
\]

\[
c + (\sum_{i=1}^{n}\alpha_{i})\left \| y \right \| - \sum_{i=1}^{n}\alpha_{i}(\delta_{i})
\]

Suppose that we choose the largest possible approximation for $\left \| z+y-x_{i} \right \|$ each term in $\sum_{i=1}^{n}\alpha_{i}(\delta_{i})$. That is, suppose that we choose $\delta_{i}$ to be as large as it can possibly be using the inequalities we obtained above. In that case, if we have that:

\[
c + (\sum_{i=1}^{n}\alpha_{i})\left \| y \right \| - \sum_{i=1}^{n}\alpha_{i}(\delta_{i}) < 0,
\]
 then it must mean that, regardless of the true value of $\left \| z+y-x_{i} \right \|$, the overall difference in the function $g$ is going to be negative because all values $\left \| z+y-x_{i} \right \| \leq \delta_{i}$

If we have that $\alpha_{k} > 0$, then we have the inequality:

\[
-\alpha_{k}(\frac{\left \|z-x_{k} \right \|^{2} + \left \langle z-x_{k}, y\right \rangle}{\left \|z-x_{k} \right \|}) \geq -\alpha_{k}\left \| z-x_{k}+y \right \| \geq
\]
\[
-\alpha_{k}\sqrt{\left \| y\right \|^{2} + \frac{(\left \|z-x_{k} \right \|^{2} + \left \langle z-x_{k}, y\right \rangle)^{2}}{\left \|z-x_{k} \right \|^{2}}},
\]
which means that the largest possible value for our estimate $\delta_{k} = \frac{\left \|z-x_{k} \right \|^{2} + \left \langle z-x_{k}, y\right \rangle}{\left \|z-x_{k} \right \|}$. On the other hand, if we have that $\alpha_{k} < 0$, then the inequality is:

\[
-\alpha_{k}\frac{\left \|z-x_{k} \right \|^{2} + \left \langle z-x_{k}, y\right \rangle}{\left \|z-x_{k} \right \|} \leq -\alpha_{k}\left \| z-x_{k}+y \right \| \leq
\]
\[
-\alpha_{k}\sqrt{\left \| y\right \|^{2} + \frac{(\left \|z-x_{k} \right \|^{2} + \left \langle z-x_{k}, y\right \rangle)^{2}}{\left \|z-x_{k} \right \|^{2}})},
\]
which means that the largest possible value of $\delta_{k} = \sqrt{\left \| y\right \|^{2} + \frac{(\left \|z-x_{k} \right \|^{2} + \left \langle z-x_{k}, y\right \rangle)^{2}}{\left \|z-x_{k} \right \|^{2}}}$

As such, we obtain the equation:

\[
c + (\sum_{i=1}^{n}\alpha_{i})\left \| y \right \| - \sum \alpha_{j} (\frac{\left \| z-x_{j} \right \|^{2} + \left \langle z-x_{j}, y \right \rangle}{\left \| z-x_{j} \right \|})  
\] 
\[
- \sum \alpha_{k} (\sqrt{\left \| y \right \|^{2} + \frac{(\left \| z-x_{k} \right \|^{2} + \left \langle z-x_{k}, y \right \rangle)^{2}}{\left \| z-x_{k} \right \|^{2}}}) < 0, \textrm{where 
 all $\alpha_{j} > 0$ and all $\alpha_{k} < 0$}
\]

As we stated earlier, this means that:

\[
\sum_{i=1}^{n}\alpha_{i}(\left \| z-x_{i} \right \| + \left \| y \right \| - \left \| z+y-x_{i} \right \|) < 0
\]

From which we obtain that:

\[
g(\left \| z-x_{1}+y \right \|, \dots, \left \| z-x_{n}+y \right \|) \geq
\]
\[
c + g(\left \| y \right \|, \dots, \left \| y \right \|)
\]

(2) On the other hand, if we follow the same process, but instead of using the largest possible estimate, we use the smallest possible estimate and have that:  

\[
c + (\sum_{i=1}^{n}\alpha_{i})\left \| y \right \| - \sum_{i=1}^{n}\alpha_{i}(\delta_{i}) > 0,
\]
then it implies that:
\[
\sum_{i=1}^{n}\alpha_{i}(\left \| z-x_{i} \right \| + \left \| y \right \| - \left \| z+y-x_{i} \right \|) > 0
\]
Because the actual value of $\left \| z-x_{k} + y \right \|$ will be smaller than the estimate we choose (note that we are changing the sign of $\left \| z-x_{k} + y \right \|$ in the equation, so an underestimate for $-\left \| z-x_{k} + y \right \|$ multiplied by $-1$ is an overestimate for $\left \| z-x_{k} + y \right \|$). Consequently:

\[
g(\left \| z+y-x_{1}\right \|, \dots, \left \| z+y-x_{n} \right \|) \leq  
\]
\[
c + g(\sqrt{\left \langle y,y \right \rangle}, \dots, \sqrt{\left \langle y,y \right \rangle}). 
\]

Since we are using the smallest possible estimate in this case, then we will have that, if $\alpha_{k} > 0$, then $\delta_{k} =\sqrt{\left \| y\right \|^{2} + \frac{(\left \|z-x_{k} \right \|^{2} + \left \langle z-x_{k}, y\right \rangle)^{2}}{\left \|z-x_{k} \right \|^{2}})}$. If $\alpha_{k} > 0$, then $\delta_{k} = \frac{\left \|z-x_{k} \right \|^{2} + \left \langle z-x_{k}, y\right \rangle}{\left \|z-x_{k} \right \|}$. Using these estimates, we get the equation:

\[
c + (\sum_{i=1}^{n}\alpha_{i})\left \| y \right \| - \sum \alpha_{j} (\sqrt{\left \| y \right \|^{2} + \frac{(\left \| z-x_{k} \right \|^{2} + \left \langle z-x_{k}, y \right \rangle)^{2}}{\left \| z-x_{k} \right \|^{2}}})  
\] 
\[
- \sum \alpha_{k} (\frac{\left \| z-x_{j} \right \|^{2} + \left \langle z-x_{j}, y \right \rangle}{\left \| z-x_{j} \right \|})  > 0, \textrm{where 
 all $\alpha_{j} > 0$ and all $\alpha_{k} < 0$}
\]

(3) However, by the Cauchy-Schwarz inequality, we also have the case that:

\[
-2 \left \| z-x_{k} \right \| \left \| y \right \| \leq 2 \left \langle z-x_{k}, y \right \rangle
\]

From this we obtain the inequality:

\[
\left \| z-x_{k} \right \| - \left \| y \right \| \leq \left \| z-x_{k}+y \right \|
\]

In this case, we would have to consider the following equation in order to determine the change in the function $g$ and whether the estimate we obtain is an overestimate or an underestimate:

\[
\sum_{i=1}^{n}\alpha_{i}(\left \| z+y-x_{i} \right \| -(\left \| z-x_{i} \right \| - \left \| y \right \|) )
\]
 
 By similar argumentation as before and realizing that the coefficient of the term $ \left \| z+y-x_{i} \right \|$ is not negative, then we can realize that if we use underestimates for $\delta_{k}$ and the result is positive, then the actual value of $\sum_{i=1}^{n}\alpha_{i}(\left \| z+y-x_{i} \right \| -(\left \| z-x_{i} \right \| - \left \| y \right \|) ) > 0$. Hence, knowing this, we obtain the equation:

 \[
(\sum_{i=1}^{n}\alpha_{i})\left \| y \right \| - c + \sum \alpha_{j} (\frac{\left \| z-x_{j} \right \|^{2} + \left \langle z-x_{j}, y \right \rangle}{\left \| z-x_{j} \right \|})  
\] 
\[
+ \sum \alpha_{k} (\sqrt{\left \| y \right \|^{2} + \frac{(\left \| z-x_{k} \right \|^{2} + \left \langle z-x_{k}, y \right \rangle)^{2}}{\left \| z-x_{k} \right \|^{2}}})  > 0, \textrm{where 
 all $\alpha_{j} > 0$ and all $\alpha_{k} < 0$}
\]

This leads to the conclusion that:

\[
g(\left \| z+y-x_{1} \right \|, \dots, \left \| z+y-x_{n} \right \|) \geq  
\]
\[
c - g(\left \| y \right \|, \dots, \left \| y \right \|). 
\]

(4) Similar to (3).
\end{proof}

\begin{remark}

The equality in $ \left \langle z-x_{i}, y \right \rangle \leq 0$ and $ \left \langle z-x_{i}, y \right \rangle \geq 0$ is left because if $ \left \langle z-x_{i}, y \right \rangle = 0$, then it is true that:

\[
\left \| z-x_{i}\right \|^2 + \left \| y\right \|^2 \leq \left \| z-x_{i}\right \|^2 + 2\leq \left \| z-x_{i}\right \|\leq \left \| y \right \| + \left \| y\right \|^{2},
\]

and

\[
\left \| z-x_{i}\right \|^2 + \left \| y\right \|^2 \geq \left \| z-x_{i}\right \|^2 - 2\leq \left \| z-x_{i}\right \|\leq \left \| y \right \| + \left \| y\right \|^{2},
\]

\end{remark}

One should note that some of these conditions can occur simultaneously. The conditions outlined above in the theorem are not mutually exclusive. Evidently, conditions (1) and (4), and (2) and (3) could be simultaneously true. That is, it could very well be the case that:

\[
c + g(\left \| y \right \|, \dots, \left \| y \right \|) \leq g(\left \| z+y-x_{1} \right \|, \dots, \left \| z+y-x_{n} \right \|) \leq c - g(\left \| y \right \|, \dots, \left \| y \right \|)
\]
or
\[
c - g(\left \| y \right \|, \dots, \left \| y \right \|) \leq g(\left \| z+y-x_{1} \right \|, \dots, \left \| z+y-x_{n} \right \|) \leq c + g(\left \| y \right \|, \dots, \left \| y \right \|)
\]
respectively. One can also notice that, through (2) and (4), and (1) and (3) cannot always be true simultaneously, they can indeed provide for better estimates. For instance, consider the case where
\[
g(\left \| z+y-x_{1} \right \|, \dots, \left \| z+y-x_{n} \right \|) \leq c - g(\left \| y \right \|, \dots, \left \| y \right \|),
\]

that (4) is true. If (2) is also true and:
\[
g(\left \| z+y-x_{1} \right \|, \dots, \left \| z+y-x_{n} \right \|) \leq c + g(\left \| y \right \|, \dots, \left \| y \right \|),
\]
then one must note the relation between $c - g(\left \| y \right \|, \dots, \left \| y \right \|)$ and $c + g(\left \| y \right \|, \dots, \left \| y \right \|)$. Evidently, if, say, $ c - g(\left \| y \right \|, \dots, \left \| y \right \|) < c + g(\left \| y \right \|, \dots, \left \| y \right \|)$, then:
\[
g(\left \| z+y-x_{1} \right \|, \dots, \left \| z+y-x_{n} \right \|) \leq c - g(\left \| y \right \|, \dots, \left \| y \right \|),
\]
has a better boundary than:
\[
g(\left \| z+y-x_{1} \right \|, \dots, \left \| z+y-x_{n} \right \|) \leq c + g(\left \| y \right \|, \dots, \left \| y \right \|),
\].

Suppose that we add a vector $y \in V$ to a vector $z \in L_{c}(x_{1}, \dots, x_{n})$. Depending on the sign of the functions detailed above in the theorem, then have that the resulting vector $z+y$ can, at most, increase $g$ by $\Delta c = g(\left \| y \right \|, \dots, \left \| y \right \|)$ or $\Delta c = -g(\left \| y \right \|, \dots, \left \| y \right \|)$ (here the values of $\alpha_{k}$ are crucial in determining this). In other words, the largest constant $k$ for which $z+y$ can be a locus point is, in fact, $k = c + \Delta c$, meaning that, at most, $z+y \in L_{c + \Delta c}(x_{1},\dots,x_{n})$. Similarly, if the values of the function are different, then the resulting vector $z+y$ can, at least, increase $g$ by $\Delta c = - g(\left \| y \right \|, \dots, \left \| y \right \|)$ or $\Delta c = g(\left \| y \right \|, \dots, \left \| y \right \|)$.

\begin{example}

 Suppose now that $V = P_{2}(\mathbb{R})$. We will define the inner product in this vector space as:

\[
\left \langle \cdot, \cdot \right \rangle \to \mathbb{R}
\]
\[
\left \langle f(x), g(x) \right \rangle = \int_{0}^{1} f(x)\cdot \overline{g(x)} dx = \int_{0}^{1} f(x)\cdot g(x) dx.
\]

Suppose now that our focus points are $x \in V$ and $-x \in V$ and our constant $c$ is $\frac{1}{2}$. Then we obtain the following equation:

\[
g(\sqrt{\left \langle ax^2+bx+c + x ,ax^2+bx+c +x \right \rangle},\]\[ \sqrt{\left \langle ax^2+bx+c - x,ax^2+bx+c - x \right \rangle}) =
\]
\[
\sqrt{\left \langle ax^2+bx+c +x ,ax^2+bx+c +x \right \rangle} -\] \[\sqrt{\left \langle ax^2+bx+c - x,ax^2+bx+c - x \right \rangle} = \frac{1}{2}.
\]

We simplify the equation even more to obtain:

\[
\sqrt{\int_{0}^{1} (ax^2+bx+c +x)^2 dx} - \sqrt{\int_{0}^{1} (ax^2+bx+c - x)^2 dx} = \frac{1}{2}.
\]

A solution to this equation is

\[
3x+\frac{\sqrt{65}}{12}-\frac{13}{6}.
\]

We will now move on to choosing a different vector $y \in V$. For example, let $y=x^2$. In this case, one can also confirm that

\[
\int_{0}^{1} (3x+\frac{\sqrt{65}}{12}-\frac{13}{6} +x)(x^2) dx \approx 0.501 > 0,
\]
and:
\[
\int_{0}^{1} (3x+\frac{\sqrt{65}}{12}-\frac{13}{6} - x)(x^2) dx \approx 0.002 > 0.
\]

In addition,

\[
c + (\sum_{i=1}^{2}\alpha_{i})\left \| y \right \| -  1 \cdot (\frac{\left \| z-x_{j} \right \|^{2} + \left \langle z-x_{j}, y \right \rangle}{\left \| z-x_{j} \right \|})  
\] 
\[
- (-1) \cdot (\sqrt{\left \| y \right \|^{2} + \frac{(\left \| z-x_{k} \right \|^{2} + \left \langle z-x_{k}, y \right \rangle)^{2}}{\left \| z-x_{k} \right \|^{2}}}) \approx
\]

\[ \frac{1}{2} - 1.6585 +0.9565  + \sqrt{10} = -0.2990 <0.\]

Hence,

\[
\sqrt{\int_{0}^{1} (3x+\frac{\sqrt{65}}{12}-\frac{13}{6} +x + x^2)^2 dx} - \sqrt{\int_{0}^{1} (3x+\frac{\sqrt{65}}{12}-\frac{13}{6} - x + x^2)^2 dx} \approx \]\[0.7868 > \frac{1}{2} + 0
\]

\end{example}

\subsubsection{The case $\alpha_{1} + \dots + \alpha_{n} = 0$}

There is a particular we might want to give more consideration to. Consider the case:

\[
g(\left \| y \right \|, \dots, \left \| y \right \|) = 0.
\]

In that case, the inequalities we proved above reduce to (depending on the values of $\alpha_{i}$, as specified above)

\[
g(\left \| z+y-x_{1} \right \|, \dots, \left \| z+y-x_{n} \right \|) \geq  
c,
\]

\[
g(\left \| z+y-x_{1} \right \|, \dots, \left \| z+y-x_{n} \right \|) \leq  
c,
\]

\[
g(\left \| z+y-x_{1} \right \|, \dots, \left \| z+y-x_{n} \right \|) \geq  
c,
\]

\[
g(\left \| z+y-x_{1} \right \|, \dots, \left \| z+y-x_{n} \right \|) \leq  
c,
\]

Moreover, note that

\[
g(\left \| y \right \|, \dots, \left \| y \right \|) = \alpha_{1} \left \| y \right \| + \dots + \alpha_{n}\left \| y \right \| = 
\]
\[
(\alpha_{1} + \dots + \alpha_{n})\left \| y \right \| = 0
\]

Then, $(\alpha_{1} + \dots + \alpha_{n})\left \| y \right \| = 0$ if and only if $\left \| y \right \| = 0$ or $\alpha_{1} + \dots + \alpha_{n} = 0$. The former equality, $\left \| y \right \| = 0$ yields that $y = 0$ by Theorem \ref{thm:theorem5}. This situation is, however, does not reveal any new information about the locus points. 

On the other hand, if we happen to have a function $g$ such that $\alpha_{1} + \dots + \alpha_{n} = 0$, then $(\alpha_{1} + \dots + \alpha_{n})\left \| y \right \| = 0$ regardless of the value of $\left \| y \right \|$. This result is much more interesting because any vector $y \in V$ can be added to any vector $z \in L_{c}(x_{1}, \dots, x_{n})$ and depending on the signs of the functions detailed in the theorem, we would have that the largest (and here by largest we understand with the highest constant) locus points where $z+y$ can be found is $ z + y \in L_{c}(x_{1}, \dots, x_{n}) $ or that the smallest locus points where $z+y$ can be found is $ z + y \in L_{c}(x_{1}, \dots, x_{n}) $). Moreover, since

\[
g(\left \| k y,ky \right \|, \dots, \left \| k y,ky \right \|) = \alpha_{1} \left \| k y,ky \right \| + \dots + \alpha_{n}\left \| k y,ky \right \| = 
\]
\[
(\alpha_{1} + \dots + \alpha_{n})\left \| k y,ky \right \| =
k(\alpha_{1} + \dots + \alpha_{n})\sqrt{\left \langle y,y \right \rangle} = k \cdot 0 = 0,
\]

any vector $ky \in V$ where $k \in \mathbb{R}$ can also be added to $z$. Thus, when this property holds, we need only focus on the original loci and the values of $\alpha_{i}$ in order to determine the direction of the inequality.

A more interesting phenomenon can be observed when $\alpha_{1} + \dots + \alpha_{n} = 0$ and we have the following equality:

\[
g(\sqrt{\left \langle z+y-x_{1}, z+y-x_{1} \right \rangle}, \dots, \sqrt{\left \langle z+y-x_{n}, z-x_{n} \right \rangle}) =  
\]
\[
 g(\sqrt{\left \langle z-x_{1}, z-x_{1} \right \rangle},\dots,\sqrt{\left \langle z-x_{n}, z-x_{n} \right \rangle}) .
\]

This equality is, perhaps, much more interesting since the vector resulting from the addition, $z+y$, lies in $ L_{c}(x_{1},\dots,x_{n})$. Such vectors $y$ are the ones that solve the equation:
\[
g(\sqrt{\left \langle z+y-x_{1}, z+y-x_{1} \right \rangle}, \dots, \sqrt{\left \langle z+y-x_{n}, z-x_{n} \right \rangle}) =  c.
\]
Each vector $z$ will have different vectors $y$ such that the equation above is satisfied.

\begin{example}

Let:

 \[
g(\sqrt{(x-2)^2+y^2}, \sqrt{(x+2)^2+y^2}) = \sqrt{(x-2)^2+y^2} - \sqrt{(x+2)^2+y^2} = 2.
 \]

 $(-2,-3)$ is a point that satisfies this equation. Let us consider all the possible vectors $y$ which can be added to this point and find the ones where the sum of both is still in $L_{2}((-2,0),(2,0))$. This might be solved using the following equation:

 \[
\sqrt{(-2+x-2)^2+(-3+y)^2} - \sqrt{(-2+x+2)^2+(-3+y)^2} = 2.
 \]

In Figure \ref{fig:graph33} we see the intersection of the graph above with the plane $z=2$, which is where the equation \[\sqrt{(-2+x-2)^2+(-3+y)^2} - \sqrt{(-2+x+2)^2+(-3+y)^2} = 2\] is satisfied, and hence we have all the vectors $y \in \mathbb{R}^2$ such that $z+y \in L_{2}((-2,0),(2,0))$. These vectors $y$ are graphed as a function in Figure \ref{fig:graph34}.

\begin{figure}
\centering
\includegraphics[width=0.75\textwidth]{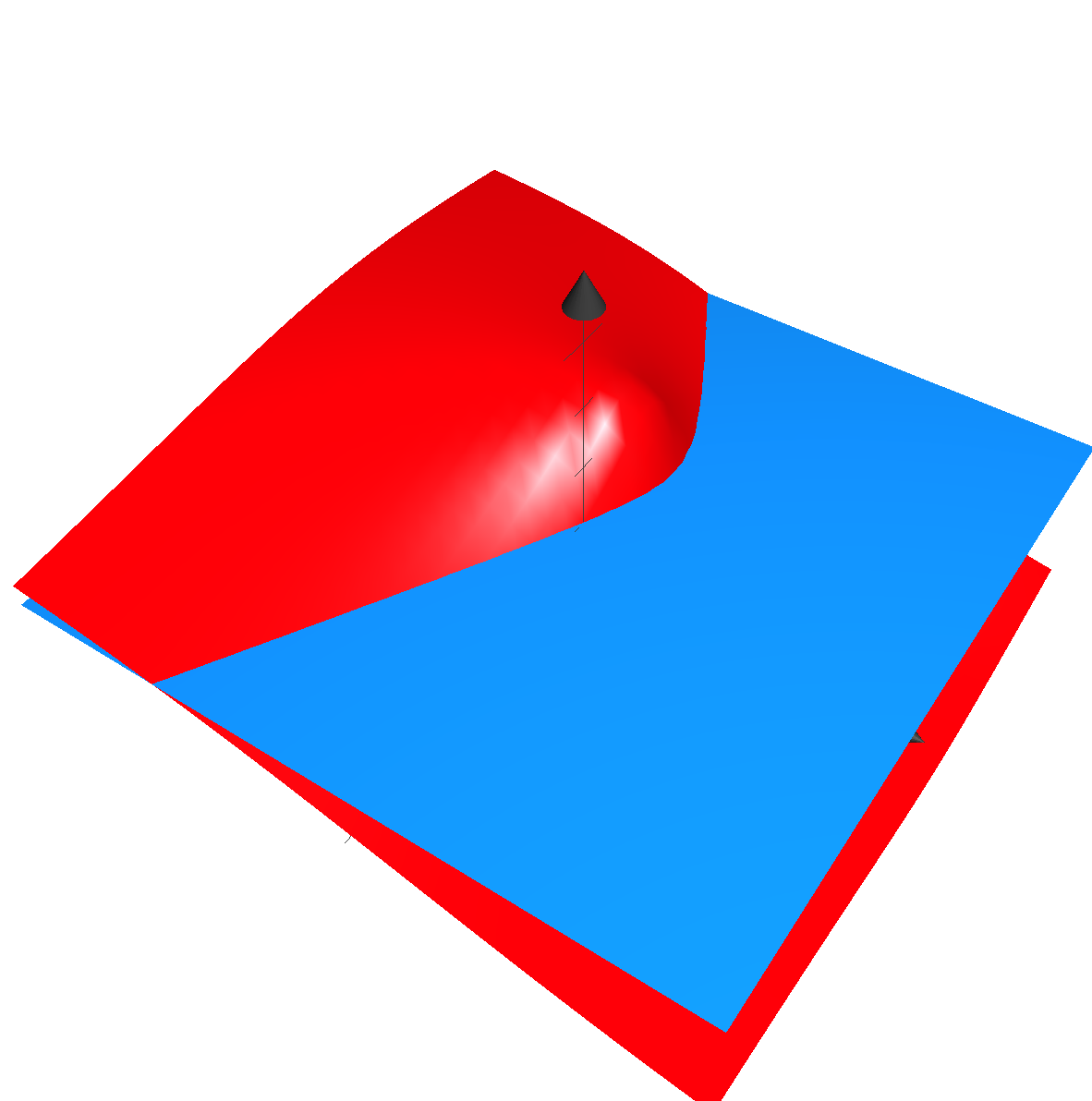}
\caption{\label{fig:graph33}Graph of $\sqrt{(-2+x-2)^2+(-3+y)^2} - \sqrt{(-2+x+2)^2+(-3+y)^2}$ and its intersection with the plane $z=2$}
\end{figure}

\begin{figure}
\centering
\includegraphics[width=0.75\textwidth]{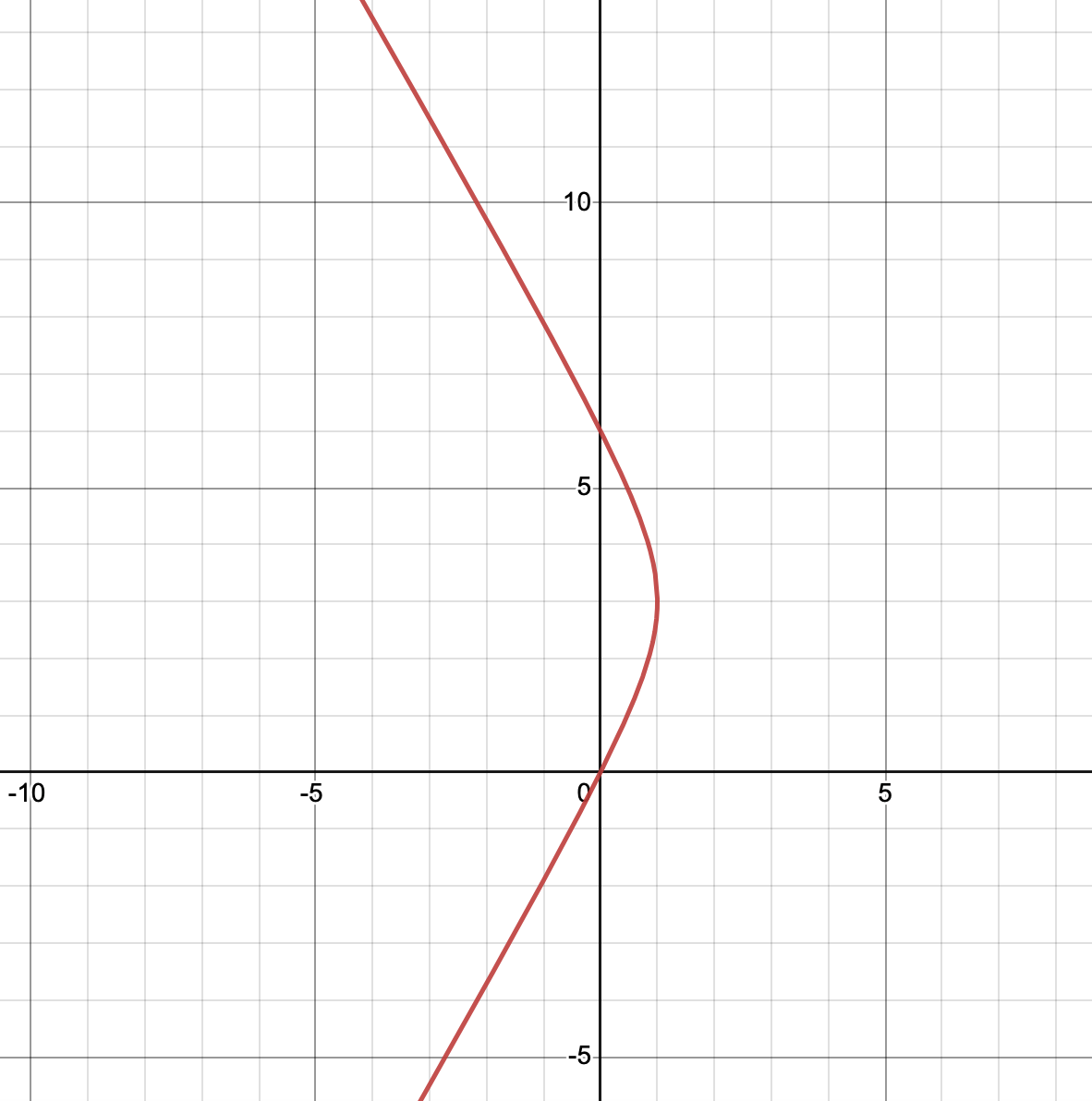}
\caption{\label{fig:graph34}Graph of $\sqrt{(-2+x-2)^2+(-3+y)^2} - \sqrt{(-2+x+2)^2+(-3+y)^2}$ and its intersection with the plane $z=2$}
\end{figure}

Now that we have these figures in mind, notice that all the points $z \in L_{c}(x_{1},\dots,x_{n})$ can be written in the form $z_{1} + y$, where $y$ is a solution to the equation above and $z_{1}$ is the point with respect to which we choose to find $y$. We can also readily realize that

\[
g(\sqrt{\left \langle z+y-x_{1}, z+y-x_{1} \right \rangle}, \dots, \sqrt{\left \langle z+y-x_{n}, z-x_{n} \right \rangle}) =  
\]
\[
g(\sqrt{\left \langle y-(x_{1}-z), y-(x_{1}-z) \right \rangle}, \dots, \sqrt{\left \langle y-(x_{n}-z), y-(x_{n}-z) \right \rangle}) .
\]

\end{example}
Essentially, what we are doing is translating the focus points to $x_{i}-z$ and then solving for the different vectors $y$. This realization is helpful because since we are translating the focus points of the function $g$, we can choose any original point $z_{1}$ that will eliminate the value of the focus points and then solve for the different values of $y$ in order to obtain the different vectors $z \in L_{c}(x_{1}, \dots,x_{n})$ in a much easier fashion.

 \pagebreak
 
\subsection{Addition of Vectors in \begin{math} L_{c}(x_{1},\dots,x_{n}) \end{math}}

Now that we have seen the behavior when trying to add some vectors to existing vectors $z \in L_{c}(x_{1},\dots,x_{n})$, and after considering the interesting case that arises specifically when we have that $\alpha_{1}+\dots+\alpha_{n}=0$, we might want to consider what happens when we are adding two vectors $v, w \in L_{c}(x_{1},\dots,x_{n})$. We saw earlier that the subset $L_{c}(x_{1},\dots,x_{n})$ is not closed under addition, hence we know that, maybe $v+w \notin L_{c}(x_{1},\dots,x_{n})$. However, we do have the following result:

\begin{theorem}
\label{thm:theorem8}
For any $v, w \in L_{c}(x_{1},\dots,x_{n})$, then

\begin{itemize}

\item[(1)] If

\[
2c - \sum \alpha_{j} (\frac{\left \| v-x_{j} \right \|^{2} + \left \langle z-x_{j}, w-x_{j} \right \rangle}{\left \| v-x_{j} \right \|})  
\] 
\[
- \sum \alpha_{k} (\sqrt{\left \| w-x_{k} \right \|^{2} + \frac{(\left \| v-x_{k} \right \|^{2} + \left \langle v-x_{k}, w-x_{k} \right \rangle)^{2}}{\left \| v-x_{k} \right \|^{2}}}) < 0, 
\] where 
 all $\alpha_{j} > 0$ and all $\alpha_{k} < 0$,
then:

\[
g(\left \| z+w-2x_{1} \right \|, \dots, \left \| v+w-2x_{n} \right \|) \geq 
2c. 
\]

\item[(2)] If 

\[
2c - \sum \alpha_{j} (\sqrt{\left \| w-x_{j} \right \|^{2} + \frac{(\left \| v-x_{j} \right \|^{2} + \left \langle v-x_{j}, w-x_{j} \right \rangle)^{2}}{\left \| v-x_{j} \right \|^{2}}})  
\] 
\[
- \sum \alpha_{k} (\frac{\left \| v-x_{j} \right \|^{2} + \left \langle v-x_{j}, w-x_{j} \right \rangle}{\left \| v-x_{j} \right \|})  > 0,
\] where 
 all $\alpha_{j} > 0$ and all $\alpha_{k} < 0$,
then:

\[
g(\left \| z+w-2x_{1} \right \|, \dots, \left \| v+w-2x_{n} \right \|) \leq 
2c.
\]

\item[(3)] If 

\[
 \sum \alpha_{j} (\frac{\left \| v-x_{j} \right \|^{2} + \left \langle v-x_{j}, w-x_{j} \right \rangle}{\left \| v-x_{j} \right \|})  
\] 
\[
+ \sum \alpha_{k} (\sqrt{\left \| w-x_{k} \right \|^{2} + \frac{(\left \| v-x_{k} \right \|^{2} + \left \langle v-x_{k}, w-x_{k} \right \rangle)^{2}}{\left \| v-x_{k} \right \|^{2}}})  > 0, 
\] where 
 all $\alpha_{j} > 0$ and all $\alpha_{k} < 0$,
then:

\[
g(\left \| z+w-2x_{1} \right \|, \dots, \left \| v+w-2x_{n} \right \|) \geq 
0.
\]

\item[(4)] If \[
\sum \alpha_{j} (\sqrt{\left \| w-x_{j} \right \|^{2} + \frac{(\left \| v-x_{j} \right \|^{2} + \left \langle v-x_{j}, w-x_{j} \right \rangle)^{2}}{\left \| v-x_{j} \right \|^{2}}})  
\] 
\[
+ \sum \alpha_{k} (\frac{\left \| v-x_{k} \right \|^{2} + \left \langle v-x_{k}, w-x_{k} \right \rangle}{\left \| v-x_{k} \right \|})  < 0,
\]where 
 all $\alpha_{j} > 0$ and all $\alpha_{k} < 0$,
 then:

\[
g(\left \| z+w-2x_{1} \right \|, \dots, \left \| v+w-2x_{n} \right \|) \geq 
0. 
\]

\end{itemize}

\end{theorem}

\begin{proof}
    We apply Theorem \ref{thm:theorem7} with $z=v$ and $y = w-x_{i}$. Note that the fact that $y$ is not a constant vector does not affect the results obtained above, since after all, the Cauchy-Schwarz inequality will still hold and the terms

\[
\sum \alpha_{j} (\sqrt{\left \| w-x_{j} \right \|^{2} + \frac{(\left \| v-x_{j} \right \|^{2} + \left \langle v-x_{j}, w-x_{j} \right \rangle)^{2}}{\left \| v-x_{j} \right \|^{2}}})  
\] 
\[
+ \sum \alpha_{k} (\frac{\left \| v-x_{k} \right \|^{2} + \left \langle v-x_{k}, w-x_{k} \right \rangle}{\left \| v-x_{k} \right \|})
\]

\[
 \sum \alpha_{j} (\frac{\left \| v-x_{j} \right \|^{2} + \left \langle v-x_{j}, w-x_{j} \right \rangle}{\left \| v-x_{j} \right \|})  
\] 
\[
+ \sum \alpha_{k} (\sqrt{\left \| w-x_{k} \right \|^{2} + \frac{(\left \| v-x_{k} \right \|^{2} + \left \langle v-x_{k}, w-x_{k} \right \rangle)^{2}}{\left \| v-x_{k} \right \|^{2}}})
\]

    still approximate the overall change of the function.
\end{proof}

This theorem above helps us understand the behavior of the addition of two vectors in the subset of locus points in a vector space. Note that the answer itself is deeply interconnected with the value of our constant $c$ and that our focus points have been altered, too. We can see, for instance, that if \[
2c - \sum \alpha_{j} (\sqrt{\left \| w-x_{j} \right \|^{2} + \frac{(\left \| v-x_{j} \right \|^{2} + \left \langle v-x_{j}, w-x_{j} \right \rangle)^{2}}{\left \| v-x_{j} \right \|^{2}}})  
\] 
\[
- \sum \alpha_{k} (\frac{\left \| v-x_{j} \right \|^{2} + \left \langle v-x_{j}, w-x_{j} \right \rangle}{\left \| v-x_{j} \right \|})  > 0,
\] where all $\alpha_{j} > 0$ and all $\alpha_{k} < 0$, then we have that $L_{2c}(2x_{1},\dots,2x_{n})$ is the largest (by largest one understands the highest possible absolute value of $2c$ and the highest norm of $2x_{i}$) subset to which the addition $v+w$ can pertain to. 

\begin{example}

Consider $V = \mathbb{R}^3$ and we let our focus points to be $(2,3,-1)$, $(-1,3,-4)$, $(0,0,3)$, $c=15$ and define the equation as

\[
g(\left \| x-x_{1} \right \|, \left \| x-x_{2} \right \|, \left \| x-x_{3} \right \|)=  
\]
\[
\sqrt{(x-2)^2+(y-3)^2+(z+1)^2}+\sqrt{(x+1)^2+(y-3)^2+(z+4)^2}
\]
\[
+\sqrt{x^2+y^2+(z-3)^2}= 15.
\]

One can easily check that
\[
(3.8945, 1, -2), (1, 2, -5.001)
\]

are points that satisfy the equation above. Now that we have these points, we can move on to verify that

\[
\left \langle (3.8945, 1, -2) - (2,3,-1), (1, 2, -5.001) - (2,3,-1) \right \rangle \approx 4.1065 > 0,
\]
\[
\left \langle (3.8945, 1, -2) - (-1,3,-4), (1, 2, -5.001) - (-1,3,-4) \right \rangle \approx 9.787 > 0,
\]
\[
\left \langle (3.8945, 1, -2) - (0,0,3), (1, 2, -5.001) - (0,0,3) \right \rangle \approx 45.8995 > 0.
\]

We also note that

\[
2c - \sum_{j=1}^{3} \alpha_{j} (\sqrt{\left \| w-x_{j} \right \|^{2} + \frac{(\left \| v-x_{j} \right \|^{2} + \left \langle v-x_{j}, w-x_{j} \right \rangle)^{2}}{\left \| v-x_{j} \right \|^{2}}})  
\] 
\[ \approx 0.2449 > 0,
\]

Hence, we have that

\[
g(\sqrt{\left \langle v+w-2x_{1}, v+w-2x_{1} \right \rangle}, \sqrt{\left \langle v+w-2x_{n}, v+w-2x_{3} \right \rangle},
\]
\[
\sqrt{\left \langle v+w-2x_{3}, v+w-2x_{3} \right \rangle})) \leq  30.
\]

Indeed, we check that

\[
\sqrt{(4.8945-4)^2+(3-6)^2+(-7.001+2)^2}+\]\[\sqrt{(4.8945+2)^2+(3-6)^2+(-7.001+8)^2}
\]
\[
+\sqrt{4.8945^2+3^2+(-7.001-6)^2}= 27.6970 \leq 30
\]

\end{example}

We may now wish to generalize this to linear combinations.

\begin{theorem}
\label{thm:theorem10}
Let $v,w \in L_{c}(x_{1},\dots,x_{n})$ and $\beta, \gamma \in \mathbb{R}_{\geq 0}$. Then:

\begin{itemize}

\item[(1)] If

\[
(\gamma+\beta)c - \sum \alpha_{j} (\frac{\left \| \gamma(v-x_{j}) \right \|^{2} + \left \langle \gamma(v-x_{j}), \beta(w-x_{j}) \right \rangle}{\left \| \gamma(v-x_{j}) \right \|})  
\] 
\[
- \sum \alpha_{k} (\sqrt{\left \| \beta(w-x_{k}) \right \|^{2} + \frac{(\left \| \gamma(v-x_{k}) \right \|^{2} + \left \langle \gamma(v-x_{k}), \beta(w-x_{j}) \right \rangle)^{2}}{\left \|\gamma(v-x_{k}) \right \|^{2}}}) < 0,\]
where all $\alpha_{j} > 0$ and all $\alpha_{k} < 0$, then:

\[
g(\left \| \gamma v+\beta w-(\gamma +\beta)x_{1} \right \|, \dots, \left \| \gamma v+\beta w-(\gamma + \beta)x_{n} \right \|) \geq 
(\gamma + \beta)c. 
\]

\item[(2)] If 

\[
(\gamma + \beta)c - \sum \alpha_{j} (\sqrt{\left \| \beta(w-x_{j}) \right \|^{2} + \frac{(\left \| \gamma(v-x_{j}) \right \|^{2} + \left \langle \gamma(v-x_{j}), \beta(w-x_{j}) \right \rangle)^{2}}{\left \| \gamma(v-x_{j}) \right \|^{2}}})  
\] 
\[
- \sum \alpha_{k} (\frac{\left \| \gamma(v-x_{k}) \right \|^{2} + \left \langle \gamma(v-x_{k}), \beta(w-x_{k}) \right \rangle}{\left \| \gamma(v-x_{k}) \right \|})  > 0, \] where all $\alpha_{j} > 0$ and all $\alpha_{k} < 0$, then:

\[
g(\left \| \gamma v+\beta w-(\gamma +\beta)x_{1} \right \|, \dots, \left \| \gamma v+\beta w-(\gamma + \beta)x_{n} \right \|) \leq 
(\gamma + \beta)c. 
\]

\item[(3)] If 

\[
(\beta-\gamma)c + \sum \alpha_{j} (\frac{\left \| \gamma(v-x_{j}) \right \|^{2} + \left \langle \gamma(v-x_{j}), \beta(w-x_{j}) \right \rangle}{\left \| \gamma(v-x_{j}) \right \|})  
\] 
\[
+ \sum \alpha_{k} (\sqrt{\left \| \beta(w-x_{k}) \right \|^{2} + \frac{(\left \| \gamma(v-x_{k}) \right \|^{2} + \left \langle \gamma(v-x_{k}), \beta(w-x_{j}) \right \rangle)^{2}}{\left \| \gamma(v-x_{k}) \right \|^{2}}})  > 0,\]
where all $\alpha_{j} > 0$ and all $\alpha_{k} < 0$, then:

\[
g(\left \| \gamma v+\beta w-(\gamma +\beta)x_{1} \right \|, \dots, \left \| \gamma v+\beta w-(\gamma + \beta)x_{n} \right \|) \geq 
(\gamma - \beta)c.  
\]

\item[(4)] If \[
(\beta-\gamma)c + \sum \alpha_{j} (\sqrt{\left \| \beta(w-x_{j}) \right \|^{2} + \frac{(\left \| \gamma(v-x_{j}) \right \|^{2} + \left \langle \gamma(v-x_{j}), \beta(w-x_{j}) \right \rangle)^{2}}{\left \| \gamma(v-x_{j}) \right \|^{2}}})  
\] 
\[
+ \sum \alpha_{k} (\frac{\left \| \gamma(v-x_{k}) \right \|^{2} + \left \langle \gamma(v-x_{k}), \beta(w-x_{k}) \right \rangle}{\left \| \gamma(v-x_{k}) \right \|})  < 0, \]where 
 all $\alpha_{j} > 0$ and all $\alpha_{k} < 0$, then

\[g(\left \| \gamma v+\beta w-(\gamma +\beta)x_{1} \right \|, \dots, \left \| \gamma v+\beta w-(\gamma + \beta)x_{n} \right \|) \leq 
(\gamma - \beta)c. \]

\end{itemize}

\end{theorem}

\begin{proof}
    We apply Theorem \ref{thm:theorem8} with $z=\gamma v$, and $y = \beta(w-x_{i})$. We realize that, since our function $g$ is linear, and so is our inner product, then we have that $ \gamma v \in L_{\gamma c}(\gamma x_{1}, \dots, \gamma x_{n})$. 
    
    If $\beta > 0$, then indeed:

    \[
    -\left \|\gamma v-\gamma x_{i} \right \| \cdot \left \| \beta(w-x_{i}) \right \| \leq \left \langle \gamma v- \gamma x_{i}, \beta(w-x_{i}) \right \rangle \leq \left \|  \gamma v- \gamma x_{i} \right \| \cdot \left \| \beta(w-x_{i}) \right \| 
    \]

    However, if $\beta < 0$, then:

    \[
    -\left \| \gamma v- \gamma x_{i} \right \| \cdot \left \| \beta(w-x_{i}) \right \| \geq \left \langle \gamma v- \gamma x_{i}, \beta(w-x_{i}) \right \rangle \geq \left \| \gamma v- \gamma x_{i} \right \| \cdot \left \| \beta(w-x_{i}) \right \| 
    \]

For that reason, if $\beta < 0$, the inequality directions change. 
 However, if $\gamma < 0$, determining how the inequality is proves to be much harder, since this also affects the term we use to approximate. For that reason, we cannot determine how the inequalities are affected in the case where $\gamma < 0$.
\end{proof}

This last theorem tells us that we can also get some information about the location of a linear combination of two vectors in $L_{c}(x_{1},\dots,x_{n})$. Similarly, we can get better boundaries for our location in the case where $\left \langle \gamma(v-x_{i}), \beta(w-x_{i}) \right \rangle \leq 0$ for all $i$ by finding the largest of $(\gamma-\beta)c$ and $(\beta - \gamma)c$ when \[
(\beta-\gamma)c + \sum \alpha_{j} (\frac{\left \| \gamma(v-x_{j}) \right \|^{2} + \left \langle \gamma(v-x_{j}), \beta(w-x_{j}) \right \rangle}{\left \| \gamma(v-x_{j}) \right \|})  
\] 
\[
+ \sum \alpha_{k} (\sqrt{\left \| \beta(w-x_{k}) \right \|^{2} + \frac{(\left \| \gamma(v-x_{k}) \right \|^{2} + \left \langle \gamma(v-x_{k}), \beta(w-x_{j}) \right \rangle)^{2}}{\left \| \gamma(v-x_{k}) \right \|^{2}}})  > 0,\]
where all $\alpha_{j} > 0$ and all $\alpha_{k} < 0$ and by finding the smallest of those when \[
(\beta-\gamma)c + \sum \alpha_{j} (\sqrt{\left \| \beta(w-x_{j}) \right \|^{2} + \frac{(\left \| \gamma(v-x_{j}) \right \|^{2} + \left \langle \gamma(v-x_{j}), \beta(w-x_{j}) \right \rangle)^{2}}{\left \| \gamma(v-x_{j}) \right \|^{2}}})  
\] 
\[
+ \sum \alpha_{k} (\frac{\left \| \gamma(v-x_{k}) \right \|^{2} + \left \langle \gamma(v-x_{k}), \beta(w-x_{k}) \right \rangle}{\left \| \gamma(v-x_{k}) \right \|})  < 0, \]

\begin{example}

Consider the real inner product space $V=P_{1}(\mathbb{R})$ and define the inner product as

\[
\left \langle f(x), g(x) \right \rangle = \frac{1}{2\pi} \int_{0}^{2\pi} f(x) g(x) dx. 
\]

We will define our equation as follows:

\[
g(\left \| f(x) - x\right \|, \left \| f(x) + 2x\right \|) = 2\left \| f(x) - x\right \| - \left \| f(x) + 2x\right \| = 4,
\]
that is,
\[
2\sqrt{\frac{1}{2\pi} \int_{0}^{2\pi} (ax+b-x)^2 dx} - \sqrt{\frac{1}{2\pi} \int_{0}^{2\pi} (ax+b+2x)^2 dx} = 4.
\]

One can easily confirm that

\[
7.9x-19.3545 \quad \text{and} \quad 3x+7.9104
\]
are solutions to the equation above. Suppose that we now let $v=4x+4.3543$ and $w=3x+7.9104$. We will take $\gamma=2$ and $\beta = 3$. Firstly, we have to take the inner products

\[
\left \langle 2(v-x_{1}), 3(w-x_{1}) \right \rangle \approx 470.1877 > 0 \quad \text{and}
\]
\[
\left \langle 2(v-x_{2}), 3(w-x_{2}) \right \rangle \approx 2641.7982 > 0.
\]

We can now proceed to calculate the value of

\[
5c -  1 \cdot (\sqrt{\left \|3(w-x_{1}) \right \|^{2} + \frac{(\left \| 2(v-x_{1}) \right \|^{2} + \left \langle 2(v-x_{1}), 3(w-x_{1}) \right \rangle)^{2}}{\left \| 2(v-x_{1}) \right \|^{2}}})  
\] 
\[
- (-1) (\frac{\left \| 2(v-x_{2}) \right \|^{2} + \left \langle 2(v-x_{2}), 3(w-x_{2}) \right \rangle}{\left \| 2(v-x_{2}) \right \|}) \]
\[
\approx 0.1973 > 0
\]

Indeed, 

\[
g(\left \| 2v +3w -5x\right \|, \left \| 2v +3w +10x\right \|) \approx 5.1431 \leq (2+3)\cdot4
\]

\end{example}

This last theorem inspires us to try to find a more generalized statement about linear combinations in general inside the real inner product. This following theorem addresses this.

\begin{theorem}
    For any vectors $v_{1}, \dots, v_{m} \in L_{c}(x_{1}, \dots, x_{n})$ and any scalars $ \beta_{1}, \dots, \beta_{m} \in \mathbb{R}^{+}$, the following hold true:

    \begin{itemize}

\item[(1)] If

\[
(\beta_{1}+\dots+\beta_{m})c - \sum \alpha_{j} (\frac{\left \| \gamma(v-x_{j}) \right \|^{2} + \left \langle \gamma(v-x_{j}), \beta(w-x_{j}) \right \rangle}{\left \| \gamma(v-x_{j}) \right \|})  
\] 
\[
- \sum \alpha_{k} (\sqrt{\left \| \beta(w-x_{k}) \right \|^{2} + \frac{(\left \| \gamma(v-x_{k}) \right \|^{2} + \left \langle \gamma(v-x_{k}), \beta(w-x_{j}) \right \rangle)^{2}}{\left \|\gamma(v-x_{k}) \right \|^{2}}}) < 0,\]
where all $\alpha_{j} > 0$ and all $\alpha_{k} < 0$, then:

\[
g(\left \| \beta_{1} \cdot v_{1} + \dots + \beta_{n} \cdot v_{n}- (\sum_{j=1}^{m} \beta_{j}) x_{1} \right \|, \dots,\] 
\[ \left \| \beta_{1} \cdot v_{1} + \dots + \beta_{n} \cdot v_{n}- (\sum_{j=1}^{m} \beta_{j}) x_{n} \right \|) \geq 
(\beta_{1}+\dots+\beta_{m})c. 
\]

\item[(2)] If 

\[
(\beta_{1} + \dots +  \beta_{m})c -\] \[\sum \alpha_{j} (\sqrt{\left \| \beta(w-x_{j}) \right \|^{2} + \frac{(\left \| \gamma(v-x_{j}) \right \|^{2} + \left \langle \gamma(v-x_{j}), \beta(w-x_{j}) \right \rangle)^{2}}{\left \| \gamma(v-x_{j}) \right \|^{2}}})  
\] 
\[
- \sum \alpha_{k} (\frac{\left \| \gamma(v-x_{k}) \right \|^{2} + \left \langle \gamma(v-x_{k}), \beta(w-x_{k}) \right \rangle}{\left \| \gamma(v-x_{k}) \right \|})  > 0, \] where all $\alpha_{j} > 0$ and all $\alpha_{k} < 0$, then:

\[
g(\left \| \beta_{1} \cdot v_{1} + \dots + \beta_{n} \cdot v_{n}- (\sum_{j=1}^{m} \beta_{j}) x_{1} \right \|, \dots,\] 
\[ \left \| \beta_{1} \cdot v_{1} + \dots + \beta_{n} \cdot v_{n}- (\sum_{j=1}^{m} \beta_{j}) x_{n} \right \|) \leq 
(\beta_{1}+\dots+\beta_{m})c. 
\]

\end{itemize}

\end{theorem}

\begin{proof}
   The proof of this theorem is similar to the proof of Theorem \ref{thm:theorem7}, although we do have to consider the fact that we are adding multiple vectors. Firstly, we have to expand and simplify the real inner product of a linear combination of vectors. For this, we will use the properties of real inner product spaces.:

   \[
   \left \langle \beta_{1} \cdot v_{1} + \dots + \beta_{n} \cdot v_{n}- (\sum_{j=1}^{m} \beta_{j}) x_{k}, \beta_{1} \cdot v_{1} + \dots + \beta_{n} \cdot v_{n}-(\sum_{j=1}^{m} \beta_{j}) x_{k} \right \rangle =
   \]
   \[
   \beta_{1} \left \langle v_{1} - x_{k}, \beta_{1} \cdot v_{1} + \dots + \beta_{n} \cdot v_{n} -(\sum_{j=1}^{m} \beta_{j}) x_{k} \right \rangle 
   \]
   \[
   +  \left \langle \beta_{2} \cdot v_{2} + \dots + \beta_{n} \cdot v_{n} - (\sum_{j=2}^{m} \beta_{j})x_{k}, \beta_{1} \cdot v_{1} + \dots + \beta_{n} \cdot v_{n} - (\sum_{j=1}^{m} \beta_{j})x_{k}\right \rangle = 
   \]
   \[
   \beta_{1}(\beta_{1}\left \langle v_{1}-x_{i}, v_{1} -x_{i}\right \rangle 
   \]
   \[
   + \left \langle \beta_{2} \cdot v_{2} + \dots + \beta_{n} \cdot v_{n} - (\sum_{j=2}^{m} \beta_{j})x_{k} , \beta_{2} \cdot v_{2} + \dots + \beta_{n} \cdot v_{n} - (\sum_{j=2}^{m} \beta_{j})x_{k} \right \rangle) +
   \]
   \[
   \left \langle \beta_{2} \cdot v_{2} + \dots + \beta_{n} \cdot v_{n} -(\sum_{j=2}^{m} \beta_{j})x_{k}, \beta_{1} \cdot v_{1} + \dots + \beta_{n} \cdot v_{n} - (\sum_{j=2}^{m} \beta_{j}) x_{k} \right \rangle
   \]

If we continue this process, we will get that the value the inner product above can be expanded and simplified to obtain:

   \[
   \left \langle \beta_{1} \cdot v_{1} + \dots + \beta_{n} \cdot v_{n}- (\sum_{j=1}^{m} \beta_{j})x_{k}, \beta_{1} \cdot v_{1} + \dots + \beta_{n} \cdot v_{n}-(\sum_{j=1}^{m} \beta_{j})x_{k} \right \rangle =
   \]
   \[
   \sum_{i=1}^{m} \sum_{j=1}^{m} \binom{2}{d} \beta_{i} \cdot \beta_{j} \left \langle v_{i}-x_{k}, v_{j} -x_{k}\right \rangle,
   \]where $d=0$ if $i=j$ and $d=1$ otherwise. 

Of course, the Cauchy-Schwarz inequality can still be applied in this case, and since we are assuming that $\beta_{1}, \dots, \beta_{m} > 0$, then:

\[
| \beta_{i} \cdot \beta_{j} \left \langle v_{i}-x_{k}, v_{j} -x_{k}\right \rangle | \leq \beta_{i} \cdot \beta_{j} \left \| v_{i}-x_{k} \right \| \left \| v_{j}-x_{k} \right \|
\]

This, of course, means that the following inequality can be found for the sum above:

\[
    \sum_{i=1}^{m} \sum_{j=1}^{m} \binom{2}{d} \beta_{i} \cdot \beta_{j} \left \langle v_{i}-x_{k}, v_{j} -x_{k}\right \rangle \leq \sum_{i=1}^{m} \sum_{j=1}^{m} \binom{2}{d} \beta_{i} \cdot \beta_{j} \left \| v_{i}-x_{k} \right \| \left \| v_{j}-x_{k} \right \|
\]

But of course, one can realize that the latter term can be expanded as:

\[
\sum_{i=1}^{m} \sum_{j=1}^{m} \binom{2}{d} \beta_{i} \cdot \beta_{j} \left \| v_{i} -x_{k}\right \| \left \| v_{j}-x_{k} \right \| =\] \[\beta_{1}^{2}\left\|v_{1}-x_{k}\right \|^{2} + 2\beta_{1} \cdot \beta_{2} \left \| v_{1} -x_{k}\right \| \left \| v_{2}-x_{k} \right \| + \dots 
\]
\[
+ 2\beta_{m} \cdot \beta_{m-1} \left \| v_{m} -x_{k}\right \| \left \| v_{m-1}-x_{k} \right \| + \beta_{m}^{2} \left \|v_{m} -x_{k}\right \|^{2}
\]
\[
 =(\beta_{1}\left\|v_{1}-x_{k}\right\|+\dots +\beta_{m}\left\|v_{m}-x_{k}\right\|)^{2}
\]

This means that:
   \[
   \left \langle \beta_{1} \cdot v_{1} + \dots + \beta_{n} \cdot v_{n}- (\sum_{j=1}^{m} \beta_{j}) x_{k}, \beta_{1} \cdot v_{1} + \dots + \beta_{n} \cdot v_{n}-(\sum_{j=1}^{m} \beta_{j}) x_{k} \right \rangle \leq
   \]
\[
(\beta_{1}\left\|v_{1}-x_{k}\right\|+\dots +\beta_{m}\left\|v_{m}-x_{k}\right\|)^{2}
\]

With this in mind, we can conclude that:

   \[
   \left \| \beta_{1} \cdot v_{1} + \dots + \beta_{n} \cdot v_{n}- (\sum_{j=1}^{m} \beta_{j}) x_{k} \right \| \leq
   \]
\[
\beta_{1}\left\|v_{1}-x_{k}\right\|+\dots +\beta_{m}\left\|v_{m}-x_{k}\right\|
\]

Knowing this, the remainder of the proof is remarkably similar to the proof of Theorem \ref{thm:theorem7}. We have to beware of the possibility that $\alpha_{k} < 0$, since that would change the direction of the inequality. However, keeping this in mind, finding the results stated above is very easy. Of course, in order to know the change in the overall function $g$, we have to consider the value of:
\[
\sum_{i=1}^{n}(\alpha_{i}(\beta_{1}(v_{1}-x_{i})+\dots + \beta_{m}(v_{m}-x_{i}))-\alpha_{i}\left \| \beta_{1}v_{1}+\dots + \beta_{m}v_{m}-(\beta_{1}+\dots + \beta_{m})x_{i}\right\|),
\]
but since $v_{1},\dots,v_{m}\in L_{c}(x_{1},\dots,x_{n})$, then the equation above simplifies to:
\[
(\beta_{1}+\dots+\beta_{m})c - \sum_{i=1}^{n}\alpha_{i}\left \| \beta_{1}v_{1}+\dots + \beta_{m}v_{m}-(\beta_{1}+\dots + \beta_{m})x_{i}\right\|
\]

Of course, we could calculate the exact value of \[\left \| \beta_{1}v_{1}+\dots + \beta_{m}v_{m}-(\beta_{1}+\dots + \beta_{m})x_{i}\right\|,\] but that would only tell us the exact answer that we are looking for, rather than a boundary. For that reason, we will make use of the formula we derived above to calculate the length as we saw above in order to approximate the length of this side. We know that:

\[
\beta_{1}v_{1}+\dots + \beta_{m}v_{m}-(\beta_{1}+\dots + \beta_{m})x_{i} = \beta_{2}v_{2}+\dots + \beta_{m}v_{m}-(\beta_{2}+\dots + \beta_{m})x_{i} + \beta_{1}(v_{1}-x_{i})
\]

Using the formula for length we derived above, we have that:

\[
\left \| \beta_{1}v_{1}+\dots + \beta_{m}v_{m}-(\beta_{1}+\dots + \beta_{m})x_{i} \right \|^{2} =
\]
\[
\left \| \beta_{2}v_{2}+\dots + \beta_{m}v_{m}-(\beta_{2}+\dots + \beta_{m})x_{i} \right \|^{2} \cdot
\]
\[
\sin^{2}(\arccos(\frac{\left\langle \beta_{2}v_{2}+\dots + \beta_{m}v_{m}-(\beta_{2}+\dots + \beta_{m})x_{i}, \beta_{1}(v_{1}-x_{i})\right \rangle }{\| \beta_{2}v_{2}+\dots + \beta_{m}v_{m}-(\beta_{2}+\dots + \beta_{m})x_{i}\| \| \beta_{1}(v_{1}-x_{i})\|}))
\]
\[
+ \frac{(\|\beta_{1}(v_{1}-x_{i})\|^{2}+\left \langle \beta_{2}v_{2}+\dots + \beta_{m}v_{m}-(\beta_{2}+\dots + \beta_{m})x_{i}, \beta_{1}(v_{1}-x_{i}) \right \rangle )^{2}}{\|\beta_{1}(v_{1}-x_{i})\|^{2}}
\]
 This process can be repeated in order to calculate the value of \[\| \beta_{2}v_{2}+\dots + \beta_{m}v_{m}-\sum_{k=2}^{m}\beta_{k}x_{i}\|,\] which we will find to be equal to:

 \[
\left \| \beta_{2}v_{2}+\dots + \beta_{m}v_{m}-\sum_{k=2}^{m}\beta_{k}x_{i} \right \|^{2} =
\]
\[
\left \| \beta_{3}v_{3}+\dots + \beta_{m}v_{m}-\sum_{k=3}^{m}\beta_{k}x_{i} \right \|^{2} \cdot
\]
\[
\sin^{2}(\arccos(\frac{\left\langle \beta_{3}v_{3}+\dots + \beta_{m}v_{m}-\sum_{k=3}^{m}\beta_{k}x_{i}, \beta_{1}(v_{1}-x_{i})\right \rangle }{\| \beta_{3}v_{3}+\dots + \beta_{m}v_{m}-\sum_{k=3}^{m}\beta_{k}x_{i}\| \| \beta_{2}(v_{2}-x_{i})\|}))
\]
\[
+ \frac{(\|\beta_{2}(v_{2}-x_{2})\|^{2}+\left \langle \beta_{3}v_{3}+\dots + \beta_{m}v_{m}-\sum_{k=3}^{m}\beta_{k}x_{i}, \beta_{2}(v_{2}-x_{i}) \right \rangle )^{2}}{\|\beta_{2}(v_{2}-x_{i})\|^{2}}
\]
We can continue doing this process iteratively for all the vectors we can, and finally we obtain the following formula for the length of the initial vector. This process is something that we will not include due to space constraints, but the calculation of the length can be easily expanded. For now, we may want to focus on the inequalities that arise from using this formula. Of course, $0 \leq \sin^{2}(\theta) \leq 1 $, and all other terms used in the computation of the length are positive. For that reason, the inequality can be easily formulated as:

\[
\frac{(\|\beta_{1}(v_{1}-x_{i})\|^{2}+\left \langle \beta_{2}v_{2}+\dots + \beta_{m}v_{m}-\sum_{k=2}^{m}\beta_{k}x_{i}, \beta_{1}(v_{1}-x_{i}) \right \rangle )^{2}}{\|\beta_{1}(v_{1}-x_{i})\|^{2}} \leq
\]
\[
\left \| \beta_{1}v_{1}+\dots + \beta_{m}v_{m}-(\beta_{1}+\dots + \beta_{m})x_{i} \right \|^{2} \leq
\]
\[
\sum_{j=1}^{m-1} (\frac{(\|\beta_{j}(v_{j}-x_{i})\|^{2}+\left \langle \beta_{j+1}v_{j+1}+\dots + \beta_{m}v_{m}-\sum_{k=j+1}^{m}\beta_{k}x_{i}, \beta_{j}(v_{j}-x_{i}) \right \rangle )^{2}}{\|\beta_{j}(v_{j}-x_{i})\|^{2}})
\]
\[
+ \| \beta_{m}(v_{m}-x_{i}) \| = \Upsilon
\]

With these inequalities in mind and noting that $\Upsilon$ is the upper bound for the length (we set the upper bound equal to $\Upsilon$ due to space constraints), we can now make use of the last part of the proof of Theorem \ref{thm:theorem7} in order to determine the direction of the inequality. Making use of the last part of said theorem, we have that if:

\[
(\beta_{1}+\dots+\beta_{m})c
\]
\[
- \sum \alpha_{j} (\frac{(\|\beta_{1}(v_{1}-x_{j})\|^{2}+\left \langle \beta_{2}v_{2}+\dots + \beta_{m}v_{m}-(\beta_{2}+\dots + \beta_{m})x_{j}, \beta_{1}(v_{1}-x_{i}) \right \rangle )}{\|\beta_{1}(v_{1}-x_{j})\|})  
\] 
\[
- \sum \alpha_{k} (\sqrt{
\Upsilon
}) < 0,\]
where all $\alpha_{j} > 0$ and all $\alpha_{k} < 0$, then we conclude that:

\[
g(\left \| \beta_{1} \cdot v_{1} + \dots + \beta_{n} \cdot v_{n}- (\sum_{j=1}^{m} \beta_{j}) x_{1} \right \|, \dots,\] 
\[ \left \| \beta_{1} \cdot v_{1} + \dots + \beta_{n} \cdot v_{n}- (\sum_{j=1}^{m} \beta_{j}) x_{n} \right \|) \geq 
(\beta_{1}+\dots+\beta_{m})c. 
\]

Similarly, if we have that:

\[
(\beta_{1} + \dots +  \beta_{m})c - \sum \alpha_{j} (\sqrt{\Upsilon})  
\] 
\[
- \sum \alpha_{k} (\frac{(\|\beta_{1}(v_{1}-x_{k})\|^{2}+\left \langle \beta_{2}v_{2}+\dots + \beta_{m}v_{m}-(\beta_{2}+\dots + \beta_{m})x_{k}, \beta_{1}(v_{1}-x_{i}) \right \rangle )}{\|\beta_{1}(v_{1}-x_{k})\|})\]  \[> 0, \] where all $\alpha_{j} > 0$ and all $\alpha_{k} < 0$, then:

\[
g(\left \| \beta_{1} \cdot v_{1} + \dots + \beta_{n} \cdot v_{n}- (\sum_{j=1}^{m} \beta_{j}) x_{1} \right \|, \dots,\] 
\[ \left \| \beta_{1} \cdot v_{1} + \dots + \beta_{n} \cdot v_{n}- (\sum_{j=1}^{m} \beta_{j}) x_{n} \right \|) \leq 
(\beta_{1}+\dots+\beta_{m})c. 
\]

\end{proof}

\begin{remark}
    It is important to note that the distinction that was mad in the hypothesis of this theorem, namely, that $\beta_{1}, \dots, \beta_{m} > 0$ was done in order to simplify the proof of the theorem. In fact, there is no restriction on what the values of $\beta_{k}$ can be. The only different that there would be in such cases is that the boundary term $(\beta_{1}+\dots+\beta_{m})c$ would become $(\beta_{1}+\dots+\beta_{v}-\beta_{v+1}-\dots-\beta_{m})$, where $\beta_{1},\dots \beta_{v} > 0$ and $\beta_{v+1}, \dots, b_{m} < 0$. This change in the boundary depending on the sing of the coefficient $\beta_{k}$ is due to the fact that, if $\beta_{l} \cdot \beta_{j} < 0$, where $\beta_{j} < 0, \beta_{l} > 0$, then:
    \[
    -\beta_{l}\beta_{j}\| v_{l}-x_{i}\| \|v_{j}-x_{i} \| \geq \beta_{l}\beta_{j}\left \langle v_{l}-x_{i}, v_{j}-x_{i} \right \rangle \geq \beta_{l}\beta_{j}\| v_{l}-x_{i}\| \|v_{j}-x_{i} \|
    \]
\end{remark}

\pagebreak

\section{Isomorphism \begin{math}V \cong \mathbb{R}^n \end{math}}

One of the nicest results in Linear Algebra is the fact that if $V$ is a $n$-dimensional vector space over $\mathbb{R}$, then $V$ is isomorphic to $\mathbb{R}^n$. One can always pick such an isomorphism to be $\phi_\beta \colon V \to \mathbb{R}^{n}$, where $\phi_\beta(x)=[x]_\beta$ with $\beta$ an ordered basis of $V$. The vector $[x]_\beta$ gathers the coordinates of $x$ with respect to $\beta$. For instance, $V=P_{2}(\mathbb{R})$ and $\beta=\{ 1, x, x^2\}$, then $\phi_\beta(ax^2+bx+c)=(c,b,a).$

This will prove to be incredibly helpful, since we will be able to use many of the properties that we previously discussed vector spaces in  $\mathbb{R}^{n}$. For example, one can consider $V=P_{1}(\mathbb{R})$, the inner product

\[
\left \langle f(x), g(x) \right \rangle = \frac{1}{2\pi} \int_{0}^{2\pi} f(x) g(x) dx 
\]
and the equation
\[
g(\left \| f(x) - x\right \|, \left \| f(x) + 2x\right \|) = 2\left \| f(x) - x\right \| - \left \| f(x) + 2x\right \| =
\]
\[
2\sqrt{\frac{1}{2\pi} \int_{0}^{2\pi} (ax+b-x)^2 dx} - \sqrt{\frac{1}{2\pi} \int_{0}^{2\pi} (ax+b+2x)^2 dx} = 4.
\]

We choose the ordered basis $\beta=\{1,x\}$. Note that $[ax+b]_{\beta} = (b,a)$. Firstly, we will try to visualize the locus points that we defined above as polynomials in $P_{1}(\mathbb{R})$ and then we will visualize them on the Cartesian plane in $\mathbb{R}^{2}$.

One can easily obtain some functions in $P_{1}(\mathbb{R})$ using the equation above. We can graph these functions and their focus points on the Cartesian plane as shown in Figure \ref{fig:graph37}. One feels that there does not seem to be a pattern between the locus points that we have graphed.

\begin{figure}
\centering
\includegraphics[width=0.75\textwidth]{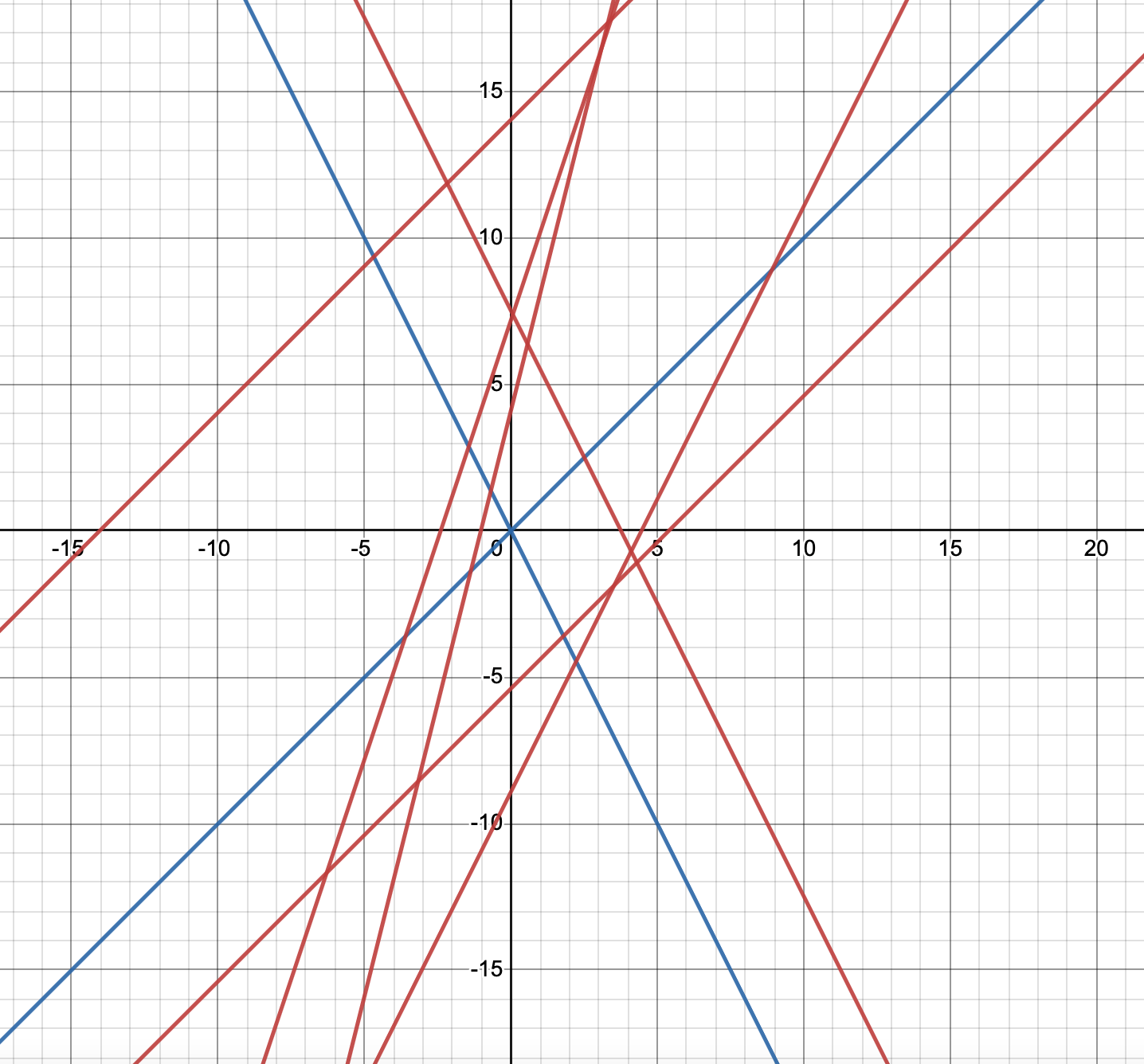}
\caption{\label{fig:graph37}Graph of some locus for the equation $2 \| f(x) - x \| - \| f(x) + 2x \| = 4$ and vector space $V=P_{1} (\mathbb{R})$ }
\end{figure}

\begin{example}

Before we step ahead and look into the isomorphism between $V$ and $\mathbb{R}^2$, we consider the vector space $\mathbb{R}^{2}$ using the same equation and its standard definition of inner product. We let the focus points to be the coordinates on the Cartesian plane of the focus points of the equation above, so the equation is 

\[
2\sqrt{x^2+(y-1)^2} - \sqrt{x^2+(y+2)^2} = 4.
\]

The graph for these locus points can be seen in Figure \ref{fig:graph38}. These locus points do not, in the slightest, resemble the ones we obtained before. Of course, one must consider that we are not applying the same inner product to both vector spaces and that the vector spaces are not the same to the one in \ref{fig:graph37}.

\begin{figure}
\centering
\includegraphics[width=0.75\textwidth]{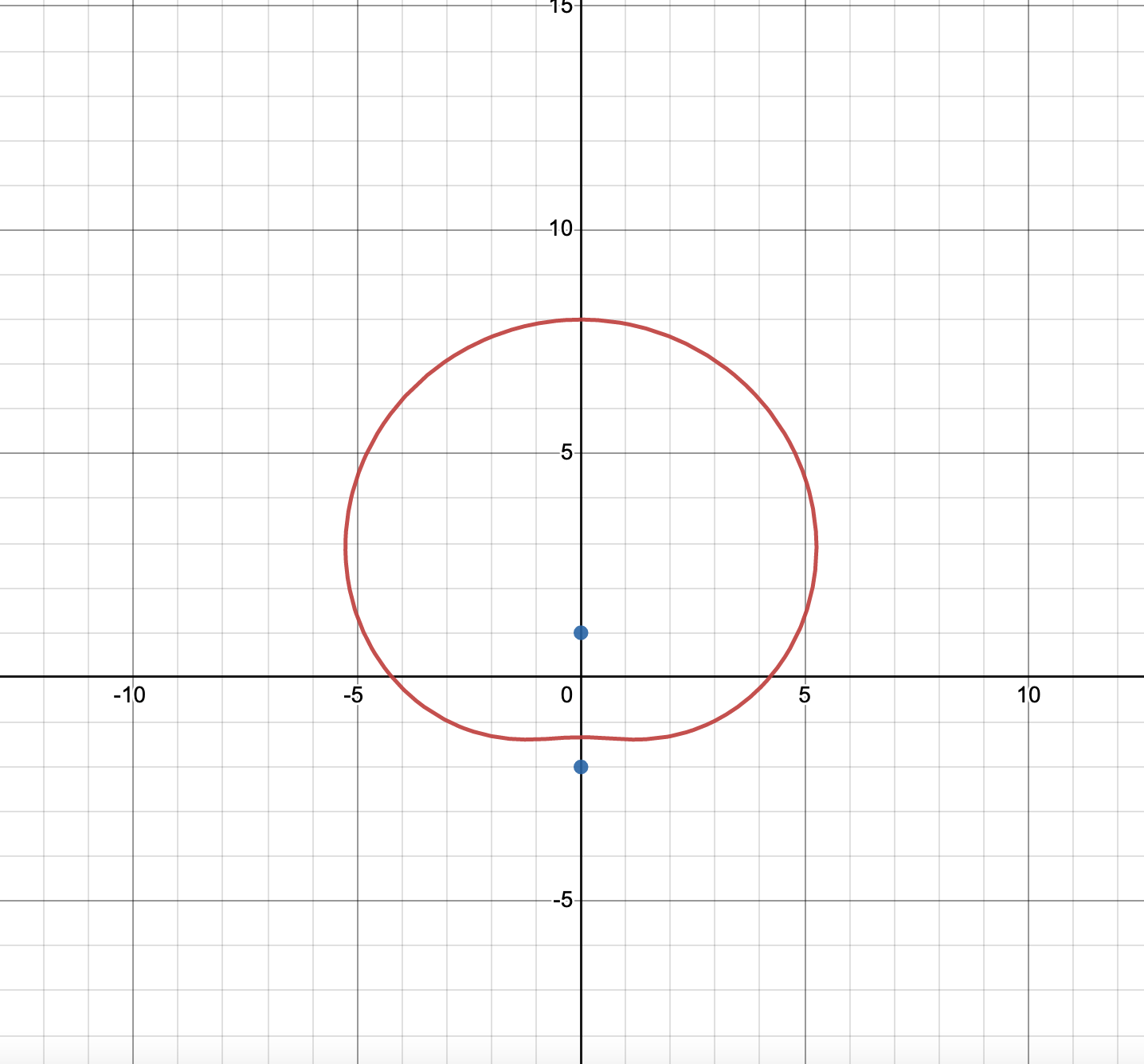}
\caption{\label{fig:graph38}Graph of the equivalent function in the vector space $V=\mathbb{R}^2$ and $2\sqrt{x^2+(y-1)^2} - \sqrt{x^2+(y+2)^2} = 4$}
\end{figure}

\end{example}

Now that we have seen a representation of the locus points by graphing the vectors in our vector space, we move on to consider the actual isomorphism between $V=P_{1}(\mathbb{R})$ and $\mathbb{R}^2$. This property proves to be incredibly helpful, since one can easily find elements of the locus in the vector space by simply inverting the isomorphism from $\mathbb{R}^{2}$ to $P_{1}(\mathbb{R})$. Not only that, but we can also work with the focus points. In order to obtain this isomorphism, we recall that $[ax+b]_{\beta} = (b,a)$. We then make use of the definition we have given to the function $g$ in order to obtain that

\[
2\sqrt{\frac{1}{2\pi} \int_{0}^{2\pi} (ax+b-x)^2 dx} - \sqrt{\frac{1}{2\pi} \int_{0}^{2\pi} (ax+b+2x)^2 dx} = 4.
\]

We can solve the integrals and leave them in terms of $a, b$ and we obtain that:

\[
\sqrt{\frac{2}{3\pi}}\sqrt{\frac{(2\pi(a-1)+b)^3-b^3}{a-1}} - \frac{\sqrt{\frac{(2\pi(a+2)+b)^3-b^3}{a+2}}}{\sqrt{6\pi}} = 4.
\]

Since we know that the isomorphism is $\phi_\beta \colon ax+ b \to (b,a) $, then we have that the equation translates to

\[
\sqrt{\frac{2}{3\pi}}\sqrt{\frac{(2\pi(y-1)+x)^3-x^3}{y-1}} - \frac{\sqrt{\frac{(2\pi(y+2)+x)^3-x^3}{y+2}}}{\sqrt{6\pi}} = 4.
\]

This is graphed in Figure \ref{fig:graph39}. The locus points that we obtain are very dissimilar from the ones we observe in Figure \ref{fig:graph38}. Evidently, this latter equation is quite different from the one we obtained by merely replicating the original equation $g$ and simply changing the vector space it applies to. This explains the differences between both locus points when both are mapped in the $\mathbb{R}^2$ plane. 
\begin{figure}
\centering
\includegraphics[width=0.75\textwidth]{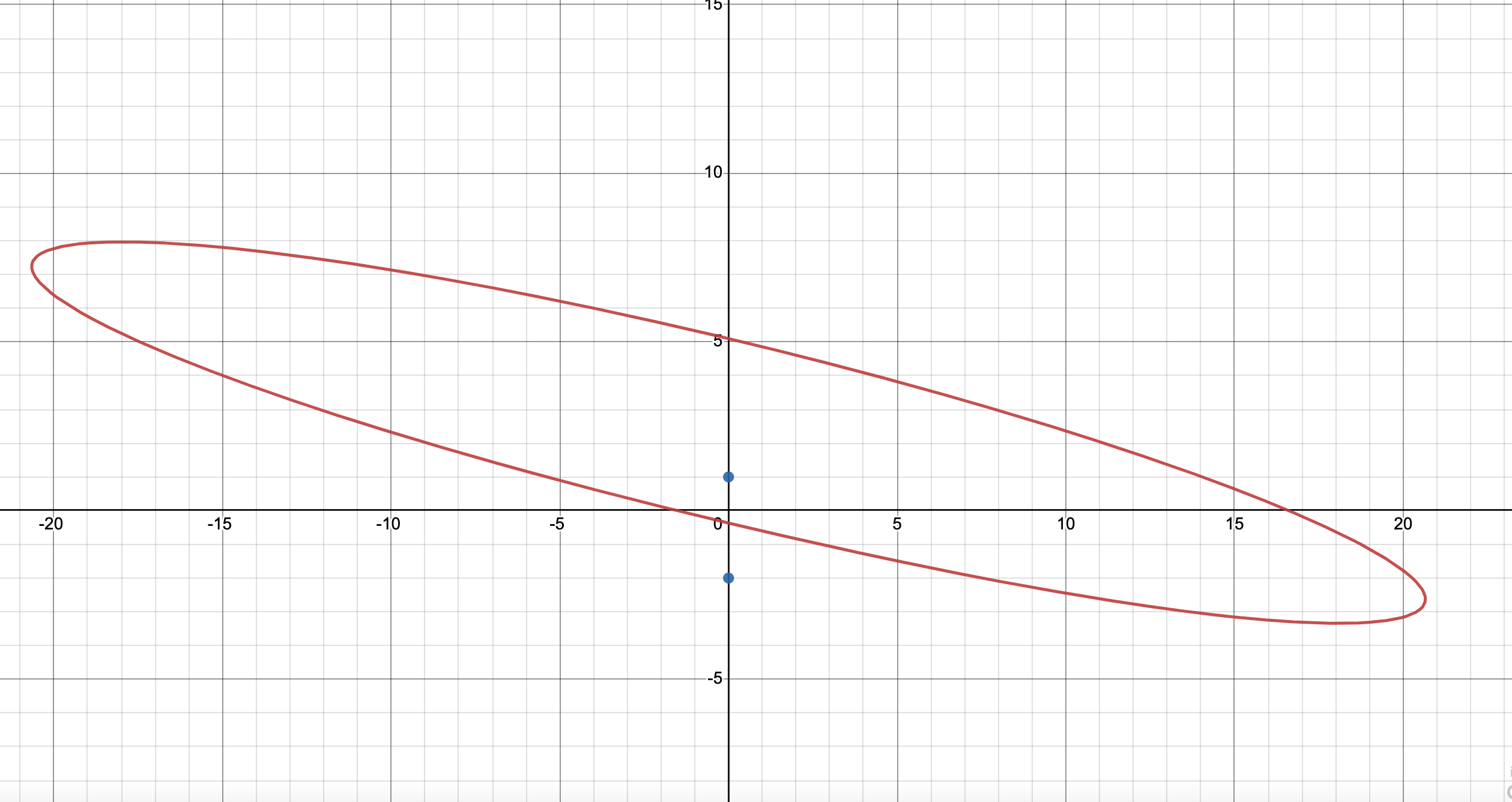}
\caption{\label{fig:graph39} Graph of the locus points of $2\left \| f(x) - x\right \| - \left \| f(x) + 2x\right \| = 4$ and vector space $V=P_{1}(\mathbb{R})$ after the isomorphism to $\mathbb{R}^2$ has been applied}
\end{figure}

Though this method does work, one could also work with the fact that the inner products in both vector spaces are defined differently and so, finding the equivalent inner product in our vector space $V$ on $\mathbb{R}^n$, can also help us find the equation. In this case one has that

\[
\left \langle (a_{1},a_{2}), (b_{1},b_{2}) \right \rangle = \frac{1}{2\pi} \int_{0}^{2\pi} (a_{2}x+a_{1})(b_{2}x+b_{1}) dx =
\]
\[
\frac{\frac{8}{3}\pi^{3}a_{2}b_{2}+2\pi^{2}a_{2}b_{1}+2\pi^{2}b_{2}a_{1}+2\pi b_{1}a_{1}}{2\pi}.
\]

This mapping of the inner product that we defined for $P_{1}(\mathbb{R})$ is evidently not the same as the standard definition of inner product in $\mathbb{R}^{2}$. However, if we now define our inner product in $\mathbb{R}^2$ as above, then we have that the norm of a vector is:

\[
\sqrt{\left \langle (a_{1},a_{2}), (a_{1},a_{2}) \right \rangle} = \sqrt{\frac{\frac{8}{3}\pi^{3}a_{2}^{2}+4\pi^{2}a_{2}a_{1} +2\pi a_{1}^{2}}{2\pi}}.
\]

If we define that to be the norm of a vector in $\mathbb{R}^2$ and we let the focus points to be the mappings on $\mathbb{R}^2$, then we have that the equation is

\[
g(\left \| x - x_{1} \right \|,\left \| x - x_{2} \right \| ) = 2\left \| x - x_{1} \right \| - \left \| x - x_{2} \right \| =
\]
\[
2  \sqrt{\frac{\frac{8}{3}\pi^{3}(y-1)^{2}+4\pi^{2}(y-1)(x) +2\pi x^{2}}{2\pi}} - \sqrt{\frac{\frac{8}{3}\pi^{3}(y+2)^{2}+4\pi^{2}(y+2)(x) +2\pi x^{2}}{2\pi}} = 4.
\]

The graph of this equation can be seen in Figure \ref{fig:graph40}. Indeed, one can easily confirm that the graphs of both equations, that is, the graph of the isomorphism of the locus points to $\mathbb{R}^2$ and the graph of the locus points using our modified definition of the norm.

The proof that both applying the isomorphism to the individual points separately and to the inner product is very straightforward.

\begin{figure}
\centering
\includegraphics[width=0.75\textwidth]{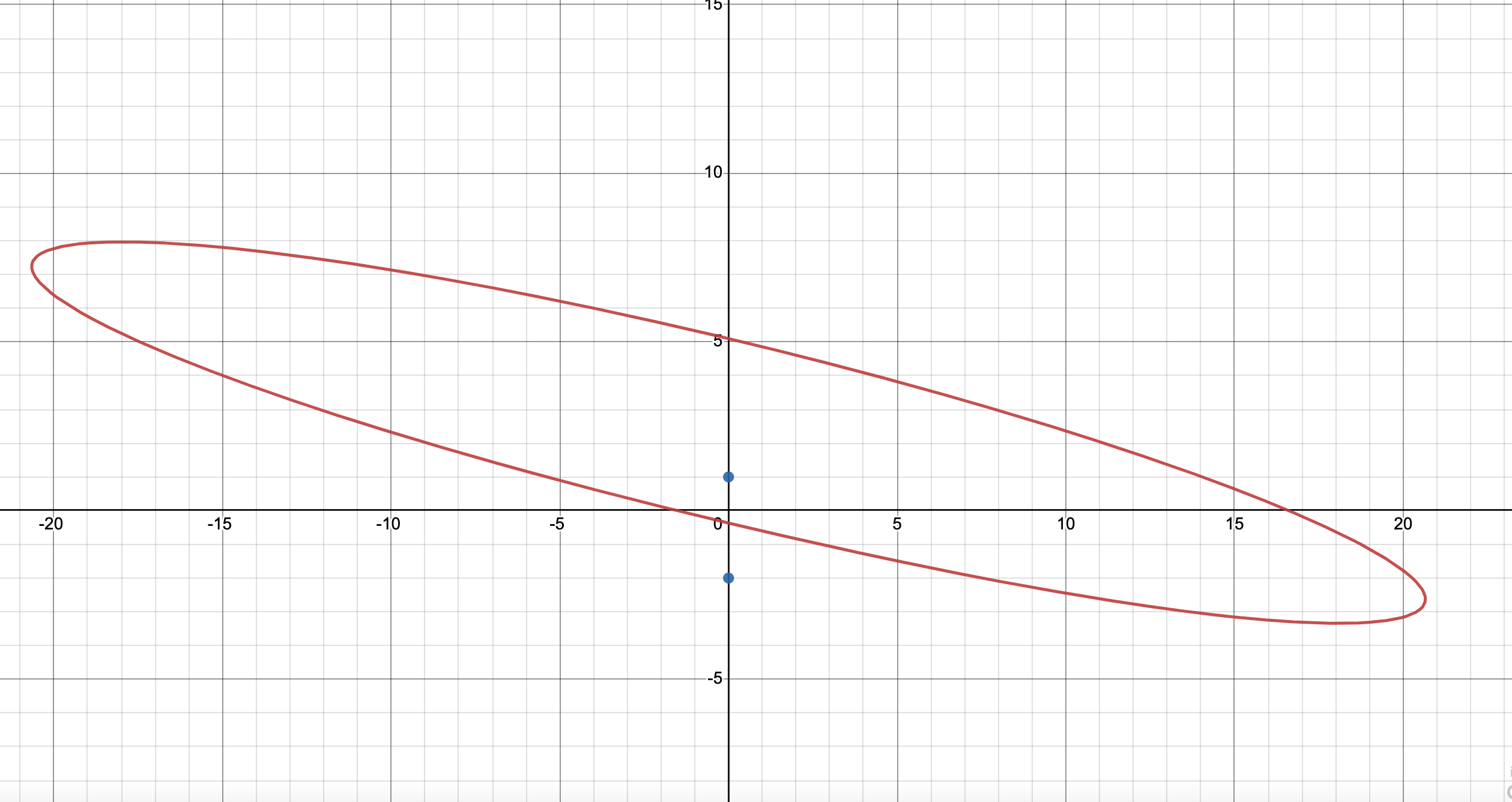}
\caption{\label{fig:graph40}Graph of the function $g(\left \| (x,y) - (0,1)\right \|, \left \| (x,y) - (0,-2)\right \|) = 4$ with the modified definition of inner product in $\mathbb{R}^2$}
\end{figure}

\subsection{Equivalence of Locus in $V$ and $\mathbb{R}^n$ under Isomorphism}
\begin{theorem}
    \label{thm:theorem12}
    Let $V$ be an $n$-dimensional vector space over $\mathbb{R}$ with inner product $\left \langle \cdot, \cdot \right \rangle \to \mathbb{R}$ and ordered basis $\beta$. Applying the isomorphism $\phi_\beta \colon V \to \mathbb{R}^n$ to the locus points $L_{c}\subseteq V$ yields the same result as finding the locus points of the equation on $\mathbb{R}^n$ with the inner product defined via the isomorphism $\phi_{\beta}$.
\end{theorem}
\begin{proof}

We will assume that we are working on a real inner product space $V$ such that $\dim V = n$ and let $\beta = \{ s_{1}, \dots, s_{n}\}$ be an ordered basis for $V$. Hence, we will be working with an isomorphism to $\mathbb{R}^{n}$. We will also define our function $g$ to be:

\[
g(\left \| x-x_{1} \right \|, \dots, \left \| x-x_{n} \right \|) = \alpha_{1} \left \| x-x_{1} \right \| + \dots + \alpha_{n} \left \| x-x_{n} \right \| = c
\]

We will also define our real inner product to be:

\[
\left \langle \cdot, \cdot \right \rangle_{V} \rightarrow \mathbb{R}
\]

What this theorem is stating is that all the locus after the isomorphism has been applied are solutions to the following equation, and that all solutions to the following equation are isomorphism to $\mathbb{R}^{n}$ of a locus:

\[
g(\left \| [x]_{\beta}-[x_{1}]_{\beta} \right \|, \dots, \left \| [x]_{\beta}-[x_{n}]_{\beta} \right \|) =\] \[\alpha_{1} \left \| [x]_{\beta}-[x_{1}]_{\beta} \right \| + \dots + \alpha_{n} \left \| [x]_{\beta}-[x_{n}]_{\beta} \right \| = c,
\]

where the variable $[x]_{\beta} \in \mathbb{R}^{n}$, the focus points $[x_{1}]_{\beta}, \dots, [x_{n}]_{\beta}$ are the isomorphisms to $\mathbb{R}^{n}$ of the original focus points, and the real inner product on the vector space $\mathbb{R}^{n}$ between any two vectors $x, y \in \mathbb{R}^{n}$ is defined to be:

\[
\left \langle x, y \right \rangle \rightarrow \mathbb{R}
\]

\[
\left \langle x, y \right \rangle \mapsto \sum_{i=1}^{n} \sum_{j=1}^{n} a_{i}b_{j} \left \langle s_{i}, s_{j} \right \rangle_{V}
\]

That is, we will define the real inner product in our vector space $\mathbb{R}^{n}$ to be the addition of all the inner products of all possible combinations of 2 elements from our basis $\beta$ times the coordinate of our $x$ vector, times the coordinate of our $y$ vector.

We will do this proof in two parts. We will firstly prove that the all the points we obtain through the isomorphism from the locus to $\mathbb{R}^{n}$ are solutions to the modified equation:

\[
g(\left \| [x]_{\beta}-[x_{1}]_{\beta} \right \|, \dots, \left \| [x]_{\beta}-[x_{n}]_{\beta} \right \|) =\] \[\alpha_{1} \left \| [x]_{\beta}-[x_{1}]_{\beta} \right \| + \dots + \alpha_{n} \left \| [x]_{\beta}-[x_{n}]_{\beta} \right \| = c,
\]

and then we will prove that all the solutions to the equation above, are also locus in the original equation:

\[
g(\left \| x-x_{1} \right \|, \dots, \left \| x-x_{n} \right \|) = \alpha_{1} \left \| x-x_{1} \right \| + \dots + \alpha_{n} \left \| x-x_{n} \right \| = c
\]

One can readily show that the real inner product in the vector space $ \left \langle \cdot, \cdot \right \rangle_{V} \rightarrow \mathbb{R}$ can be transformed to be a real inner product in the vector space $ \mathbb{R}^{n}$. Let $v, w \in V$. We also know that $\beta$ is an ordered basis for $V$, hence $v = a_{1}\cdot s_{1}+ \dots a_{n} s_{n}$ and $ w = b_{1}\cdot s_{1}+ \dots b_{n} s_{n}$. With this in mind, we can write the inner product between these two vectors in $V$ as follows:

\[
\left \langle v, w \right \rangle_{V} = \left \langle a_{1}\cdot s_{1}+ \dots a_{n} s_{n}, b_{1}\cdot s_{1}+ \dots b_{n} s_{n} \right \rangle_{V}
\]

But we know that real inner products are linear, so we inner product above simplifies to:

\[
\left \langle a_{1}\cdot s_{1}+ \dots a_{n} s_{n}, b_{1}\cdot s_{1}+ \dots b_{n} s_{n} \right \rangle_{V} = a_{1}b_{1}\left \langle s_{1}, s_{1} \right \rangle_{V} 
\]
\[
+ \dots + a_{n}b_{1}\left \langle s_{n}, s_{1} \right \rangle_{V} + \dots + a_{1}b_{n}\left \langle s_{1}, s_{n} \right \rangle_{V} + \dots a_{n}b_{n}\left \langle s_{n}, s_{n} \right \rangle_{V} =
\]
\[
\sum_{i=1}^{n} \sum_{j=1}^{n} a_{i}b_{j} \left \langle s_{i}, s_{j} \right \rangle_{V}
\]

If we now let $x = [v]_{\beta}$ and $y = [w]_{\beta}$, then we have a definition for a real inner product of vectors $x, y \in \mathbb{R}^{n}$.

Now that we have proved that these two inner products are equivalent after applying the isomorphism, we will prove that the isomorphism of a vector $z \in L_{c}(x_{1},\dots, x_{n})$ is a solution to the equation:

\[
g(\left \| [x]_{\beta}-[x_{1}]_{\beta} \right \|, \dots, \left \| [x]_{\beta}-[x_{n}]_{\beta} \right \|) =\] \[\alpha_{1} \left \| [x]_{\beta}-[x_{1}]_{\beta} \right \| + \dots + \alpha_{n} \left \| [x]_{\beta}-[x_{n}]_{\beta} \right \| = c,
\]

Of course, since $z \in L_{c}(x_{1},\dots, x_{n})$, then:

\[
g(\left \| z-x_{1} \right \|, \dots, \left \| z-x_{n} \right \|) = \alpha_{1} \left \| z-x_{1} \right \| + \dots + \alpha_{n} \left \| z-x_{n} \right \| = c
\]

Since vector spaces are closed under linear combinations, then we have that $z-x_{1}, \dots, z-x_{n} \in V$. Knowing that $\beta$ is an ordered basis for $V$, then we can write those vectors as:

\[
\left\{\begin{matrix}
z-x_{1}= (a_{1} - a_{x_{1}1})s_{1} +  \dots (a_{n} - a_{x_{1}n})s_{n}\\ 
\dots\\ 
z-x_{n} = (a_{1} - a_{x_{n}1})s_{1} +  \dots (a_{n} - a_{x_{n}n})s_{n}
\end{matrix}\right.
\]

Where $a_{11}, \dots a_{1n}$ are the coefficients of $z$ and $a_{x_{1}1}, \dots, a_{x_{n}n}$ are the coefficients for the different focus points $x_{i}$. In proving the equivalence between the inner product in a vector space $V$ and $\mathbb{R}^{n}$, we realized that:

\[
\left \langle z-x_{1}, z-x_{1} \right \rangle_{V} = \sum_{i=1}^{n} \sum_{j=1}^{n} (a_{i} - a_{x_{1}i})(a_{j} - a_{x_{1}j}) \left \langle s_{i}, s_{j} \right \rangle_{V}
\]

We can do this to all terms, and we have that the equation $g$ becomes:

\[
g(\left \| z-x_{1} \right \|, \dots, \left \| z-x_{n} \right \|) = \alpha_{1} \left \| z-x_{1} \right \| + \dots + \alpha_{n} \left \| z-x_{n} \right \| = c
\]
\[
\alpha_{1} \sqrt{\sum_{i=1}^{n} \sum_{j=1}^{n} (a_{i} - a_{x_{1}i})(a_{j} - a_{x_{1}j}) \left \langle s_{i}, s_{j} \right \rangle_{V}} + \dots +\] \[\alpha_{n}\sqrt{ \sum_{j=1}^{n} (a_{i} - a_{x_{n}i})(a_{j} - a_{x_{n}j}) \left \langle s_{i}, s_{j} \right \rangle_{V}}
\]

Knowing that $[z]_{\beta} = (a_{1}, \dots, a_{n})$ and that $[x_{i}]_{\beta} = (a_{x_{i}1}, \dots, a_{x_{i}n})$ and having in mind that we defined the inner product in $\mathbb{R}^{n}$ to be $\left \langle x, y \right \rangle \mapsto \sum_{i=1}^{n} \sum_{j=1}^{n} a_{i}b_{j} \left \langle s_{i}, s_{j} \right \rangle_{V}$, then we have that:

\[
\alpha_{1} \sqrt{\sum_{i=1}^{n} \sum_{j=1}^{n} (a_{i} - a_{x_{1}i})(a_{j} - a_{x_{1}j}) \left \langle s_{i}, s_{j} \right \rangle_{V}} + \dots + \alpha_{n} \sqrt{ \sum_{j=1}^{n} (a_{i} - a_{x_{n}i})(a_{j} - a_{x_{n}j}) \left \langle s_{i}, s_{j} \right \rangle_{V}} =
\]
\[
\alpha_{1} \sqrt{\left \langle [z]_{\beta} - [x_{1}]_{\beta}, [z]_{\beta} - [x_{1}]_{\beta} \right \rangle} + \dots +\] \[\alpha_{n} \sqrt{ \left \langle [z]_{\beta} - [x_{n}]_{\beta}, [z]_{\beta} - [x_{n}]_{\beta} \right \rangle} =
\]
\[
g(\left \| [z]_{\beta}-[x_{1}]_{\beta} \right \|, \dots, \left \| [z]_{\beta}-[x_{n}]_{\beta} \right \|)
\]

Hence, we have proved that if a vector $z \in L_{c}(x_{1}, \dots, x_{n})$, then its isomorphism to $\mathbb{R}^{n}$ is a solution to the equation:

\[
g(\left \| [x]_{\beta}-[x_{1}]_{\beta} \right \|, \dots, \left \| [x]_{\beta}-[x_{n}]_{\beta} \right \|) = c,
\]
which is to say that $[z]_{\beta} \in L_{c}([x_{1}]_{\beta}, \dots, [x_{n}]_{\beta})$.

Now we will prove the other direction. That is, if we have that $[x]_{\beta} \in L_{c}([x_{1}]_{\beta}, \dots, [x_{n}]_{\beta})$, then $x \in L_{c}(x_{1}, \dots, x_{n})$.

Suppose that $[x]_{\beta} \in L_{c}([x_{1}]_{\beta}, \dots, [x_{n}]_{\beta})$. Then we know that:

\[
g(\left \| [x]_{\beta}-[x_{1}]_{\beta} \right \|, \dots, \left \| [x]_{\beta}-[x_{n}]_{\beta} \right \|) = \alpha_{1} \left \| [x]_{\beta}-[x_{1}]_{\beta} \right \| + \dots + \alpha_{n} \left \| [x]_{\beta}-[x_{n}]_{\beta} \right \| = c,
\]

This means that:

\[
\alpha_{1} \sqrt{\left \langle [x]_{\beta} - [x_{1}]_{\beta}, [x]_{\beta} - [x_{1}]_{\beta} \right \rangle} + \dots + \alpha_{n}  \sqrt{\left \langle [x]_{\beta} - [x_{n}]_{\beta}, [x]_{\beta} - [x_{n}]_{\beta} \right \rangle} = c
\]

Since we defined the real inner product in $\mathbb{R}^{n}$ to be the equation given above, then we can substitute that in for the function and, letting $[x]_{\beta} = (b_{1}, \dots, b_{n})$ we have that:

\[
\alpha_{1} \sqrt{\left \langle [x]_{\beta} - [x_{1}]_{\beta}, [x]_{\beta} - [x_{1}]_{\beta} \right \rangle} + \dots + \alpha_{n}  \sqrt{\left \langle [x]_{\beta} - [x_{n}]_{\beta}, [x]_{\beta} - [x_{n}]_{\beta} \right \rangle} = c =
\]
\[
\alpha_{1} \sqrt{\sum_{i=1}^{n} \sum_{j=1}^{n} (b_{i} - a_{x_{1}i})(b_{j} - a_{x_{1}j}) \left \langle s_{i}, s_{j} \right \rangle_{V}} + \dots +\] \[\alpha_{n}\sqrt{ \sum_{j=1}^{n} (b_{i} - a_{x_{n}i})(b_{j} - a_{x_{n}j}) \left \langle s_{i}, s_{j} \right \rangle_{V}}
\]

In proving the equivalence between the inner product on $V$ and this inner product in $\mathbb{R}$, we realized that 

\[
\left \langle \sum_{i=1}^{n} (b_{i} - a_{x_{1}i})s_{i}, \sum_{i=1}^{n} (b_{i} - a_{x_{1}i})s_{i} \right \rangle_{V} = sum_{i=1}^{n} \sum_{j=1}^{n} (b_{i} - a_{x_{1}i})(b_{j} - a_{x_{1}j}) \left \langle s_{i}, s_{j} \right \rangle_{V}
\]

But, of course, $ \sum_{i=1}^{n} b_{i} \cdot s_{i} = x \in V$ and $\sum_{i=1}^{n} a_{x_{1}i} \cdot s_{i} = x_{1} \in V$. We can do this for all the inner products in the equation above and we have that:

\[
\alpha_{1} \sqrt{\sum_{i=1}^{n} \sum_{j=1}^{n} (b_{i} - a_{x_{1}i})(b_{j} - a_{x_{1}j}) \left \langle s_{i}, s_{j} \right \rangle_{V}} + \dots +\] \[\alpha_{n}\sqrt{ \sum_{j=1}^{n} (b_{i} - a_{x_{n}i})(b_{j} - a_{x_{n}j}) \left \langle s_{i}, s_{j} \right \rangle_{V}} =
\]
\[
\alpha_{1} \sqrt{\left \langle x-x_{1}, x-x_{1}\right \rangle_{V}} + \dots + \alpha_{n}\sqrt{\left \langle x-x_{1}, x-x_{1}\right \rangle_{V}} = 
\]
\[
\alpha_{1} \left \| x-x_{1} \right \| + \dots + \alpha_{n} \left \| x-x_{n} \right \| = g(\left \| x-x_{1} \right \|, \dots, \left \| x-x_{n} \right \|) = c
\]

Hence we have proven that if $[x]_{\beta}$ is a solution to the equation:

\[
g(\left \| [x]_{\beta}-[x_{1}]_{\beta} \right \|, \dots, \left \| [x]_{\beta}-[x_{n}]_{\beta} \right \|) = \alpha_{1} \left \| [x]_{\beta}-[x_{1}]_{\beta} \right \| + \dots + \alpha_{n} \left \| [x]_{\beta}-[x_{n}]_{\beta} \right \| = c,
\]
then $x$ is a solution to the equation:

\[ g(\left \| x-x_{1} \right \|, \dots, \left \| x-x_{n} \right \|) = c
\]

\end{proof}

This last theorem tells us that the order in which we apply the function or the isomorphism does not matter, so long as we have an appropriate definition for the real inner product in the vector space $\mathbb{R}^{n}$, which we define to be:

\[
\left \langle x, y \right \rangle \rightarrow \mathbb{R}
\]

\[
\left \langle x, y \right \rangle \mapsto \sum_{i=1}^{n} \sum_{j=1}^{n} a_{i}b_{j} \left \langle s_{i}, s_{j} \right \rangle_{V}
\]

This insight is incredibly helpful because in many cases it can be computationally easier to calculate the values of the locus in $\mathbb{R}^{n}$ and to then apply the inverse of the isomorphism in order to find the locus in the vector space $V$. In addition to that, this isomorphism can, perhaps, give more insight into how the locus looks and how it maps to $\mathbb{R}^{n}$.

As an example, we will consider the definition of an ellipse in 3 different vector spaces with and isomorphism to $\mathbb{R}^{2}$. An ellipse is defined to be the set of all points such that, given two fixes focus points, the sum of the distances between these focus points to the points in the ellipse is constant. What this means for our function $g$ is that it must be of the form:

\[
g(\left \| x-x_{1}\right \|,\left \| x-x_{2}\right \|) = \left \| x-x_{1}\right \| + \left \| x-x_{2}\right \| = c
\]

\begin{example}
    
Consider the following vector spaces, their inner products, and their ordered bases:

\[
\left\{\begin{matrix}
V = M_{2 \times 1}(\mathbb{R})\\ 
\left \langle A, B \right \rangle_{V} = \mathrm{tr}(A^{T}\cdot B) \\
\beta = \{ \begin{bmatrix} 1\\  0\end{bmatrix}, \begin{bmatrix} 0\\  1\end{bmatrix}\}
\end{matrix}\right.
\]
\[
\left\{\begin{matrix}
U = P_{1}(R)\\ 
\left \langle f(x), g(x) \right \rangle_{U} = \int_{0}^{1} f(x) \cdot g(x)  dx \\
\gamma = \{ 1, x\}

\end{matrix}\right.
\]
\[
\left\{\begin{matrix}
W = \mathbb{R}^{2}\\ 
\left \langle (a_{1},a_{2}), (b_{1},b_{2}) \right \rangle_{W} = a_{1} \cdot b_{1} +a_{2} \cdot b_{2} \\
\zeta = \{ (1,0), (0,1) \}
\end{matrix}\right.
\]

Now that we have the vector spaces, inner products, and bases defined, we will define our functions for an ellipse to be:

\[
g(\left \| X-\begin{bmatrix} 0\\  2\end{bmatrix}\right \|,\left \| X-\begin{bmatrix} 0\\  -2\end{bmatrix}\right \|) = 10
\]
\[
g(\left \| f(x)-2x\right \|,\left \| f(x)+2x\right \|) = 10
\]
\[
g(\left \| (x,y)-(0,2)\right \|,\left \| (x,y)-(0,-2)\right \|) = 10
\]

Note that the isomorphisms of the focus points to $\mathbb{R}^{2}$ are the same. That is, $[\begin{bmatrix} 0\\  2\end{bmatrix}]_{\beta} = [2x]_{\gamma} = [(0,2)]_{\zeta} = (0,2)$, and $[\begin{bmatrix} 0\\  -2\end{bmatrix}]_{\beta} = [-2x]_{\gamma} = [(0,-2)]_{\zeta} = (0,-2)$

Firstly, we need to note how the vector spaces are spanned with respect to their bases. We have that the vector spaces can be spanned as follows and, hence, their coordinates with respect to their bases can be written as follows too:

\[
V = \mathrm{span}(\{\begin{bmatrix} 1\\  0\end{bmatrix}, \begin{bmatrix} 0\\  1\end{bmatrix}\}) = a_{1} \cdot \begin{bmatrix} 1\\  0\end{bmatrix} + a_{2} \cdot \begin{bmatrix} 0\\  1\end{bmatrix} = \begin{bmatrix} a_{1}\\  a_{2}\end{bmatrix}
\]
\[
[\begin{bmatrix} a_{1}\\  a_{2}\end{bmatrix}]_{\beta} = (a_{1},a_{2})
\]

\[
U = \mathrm{span}(\{1, x \}) = a_{1} \cdot 1 + a_{2} \cdot x = a_{1} + a_{2}x
\]
\[
[ a_{1} +  a_{2}x]_{\gamma} = (a_{1},a_{2})
\]

Now that we have a clear idea of how the vectors in the vector space map to the vector space $\mathbb{R}^{2}$, we can move on to finding the real inner product of these mappings using the inner products for its vector space. We will start by working with the vector space $V$

\[
\left \langle A, B \right \rangle_{V} = \mathrm{tr}(A^{T}\cdot B) 
\]
\[
A = \begin{bmatrix} a_{1}\\  a_{2}\end{bmatrix}, B = \begin{bmatrix} b_{1}\\  b_{2}\end{bmatrix}
\]
\[
[A]_{\beta} =(a_{1},a_{2}), [B]_{\beta} =(b_{1},b_{2})
\]
\[
\left \langle (a_{1}, a_{2}), (b_{1}, b_{2}) \right \rangle = \sum_{i=1}^{2} \sum_{j=1}^{2} a_{i}b_{j} \left \langle s_{i}, s_{j} \right \rangle_{V} =
\]
\[
a_{1}b_{1} \left \langle \begin{bmatrix} 1\\  0\end{bmatrix}, \begin{bmatrix} 1\\  0\end{bmatrix} \right \rangle_{V} + a_{1}b_{2} \left \langle \begin{bmatrix} 1\\  0\end{bmatrix}, \begin{bmatrix} 0\\  1\end{bmatrix} \right \rangle_{V} 
\]
\[
+ a_{2}b_{1} \left \langle \begin{bmatrix} 0\\  1\end{bmatrix}, \begin{bmatrix} 1\\  0\end{bmatrix} \right \rangle_{V} + a_{2}b_{2} \left \langle \begin{bmatrix} 0\\  1\end{bmatrix}, \begin{bmatrix} 0\\  1\end{bmatrix} \right \rangle_{V} =
\]
\[
a_{1}b_{1}\cdot 1 + a_{1}b_{2} \cdot 0 + a_{2}b_{1} \cdot 0 + a_{2}b_{2} \cdot 1 = a_{1}b_{1} + a_{2}b_{2}
\]

This means that the norm of a vector in the vector space $V$ is given by:

\[
\sqrt{\left \langle (a_{1}, a_{2}), (a_{1}, a_{2}) \right \rangle} = \sqrt{a_{1}^{2}+a_{2}^{2}}
\]

With this in mind, we know, hence, that the isomorphism of the locus in the vector space $V$ with its real inner product, can also be found as the locus of the equation:

\[
g(\left \| (x,y)-(0,2)\right \|,\left \| (x,y)-(0,-2)\right \|) = 
\]
\[
\sqrt{x^{2}+(y-2)^{2}} + \sqrt{x^{2}+(y+2)^{2}} = 10
\]

Moving on, we have the vector space $U$. We will follow the same procedure to find the definition of real inner product in the vector space $\mathbb{R}^{2}$ such that the isomorphism of the loci of $U$ is the loci of $g(\left \| (x,y)-(0,2)\right \|,\left \| (x,y)-(0,-2)\right \|)$, with an altered definition of norm.

\[
\left \langle f(x), g(x) \right \rangle_{U} = \int_{0}^{1} f(x) \cdot g(x)  dx
\]
\[
f(x)=a_{1}+a_{2}x, g(x)=b_{1}+b_{2}x
\]
\[
[f(x)]_{\gamma}=(a_{1},a_{2}), [g(x)]_{\gamma}=(b_{1},b_{2})
\]
\[
\left \langle (a_{1}, a_{2}), (b_{1}, b_{2}) \right \rangle = \sum_{i=1}^{2} \sum_{j=1}^{2} a_{i}b_{j} \left \langle s_{i}, s_{j} \right \rangle_{U} =
\]
\[
a_{1}b_{1} \left \langle 1,1 \right \rangle_{U} + a_{1}b_{2} \left \langle 1, x \right \rangle_{U} + a_{2}b_{1} \left \langle x, 1 \right \rangle_{U} + a_{2}b_{2} \left \langle x,x \right \rangle_{U} =
\]
\[
a_{1}b_{1} \cdot 1 + a_{1}b_{2} \cdot \frac{1}{2} + a_{2}b_{1} \cdot \frac{1}{2} + a_{2}b_{2} \cdot \frac{1}{3} = 
\]
\[
\frac{6 a_{1}b_{1} + 3 a_{1}b_{2} + 3 a_{2}b_{1} + 2 a_{2}b_{2}}{6}
\]

This means that the norm of a vector in the vector space $U$ is given by:

\[
\sqrt{\left \langle (a_{1}, a_{2}), (a_{1}, a_{2}) \right \rangle} = \sqrt{\frac{6a_{1}^2+6a_{1}a_{2}+2a_{2}^2}{6}}
\]

With this in mind, we can determine that the isomorphism of the locus in $U$ can be obtained by solving the equation:

\[
g(\left \| (x,y)-(0,2)\right \|,\left \| (x,y)-(0,-2)\right \|) = 
\]
\[
\sqrt{\frac{6x^2+6x(y-2)+2(y-2)^2}{6}} + \sqrt{\frac{6x^2+6x(y+2)+2(y+2)^2}{6}} = 10
\]

\begin{figure}
\centering
\includegraphics[width=0.75\textwidth]{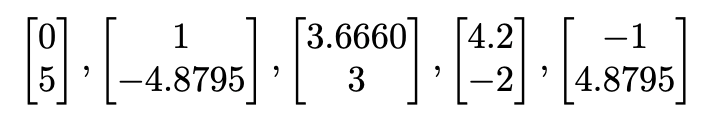}
\caption{\label{fig:graph46}Subset of the locus of $g(\left \| X-\begin{bmatrix} 0\\  2\end{bmatrix}\right \|,\left \| X-\begin{bmatrix} 0\\  -2\end{bmatrix}\right \|) = 10$ in the vector space $V$}
\end{figure}

\begin{figure}
\centering
\includegraphics[width=0.75\textwidth]{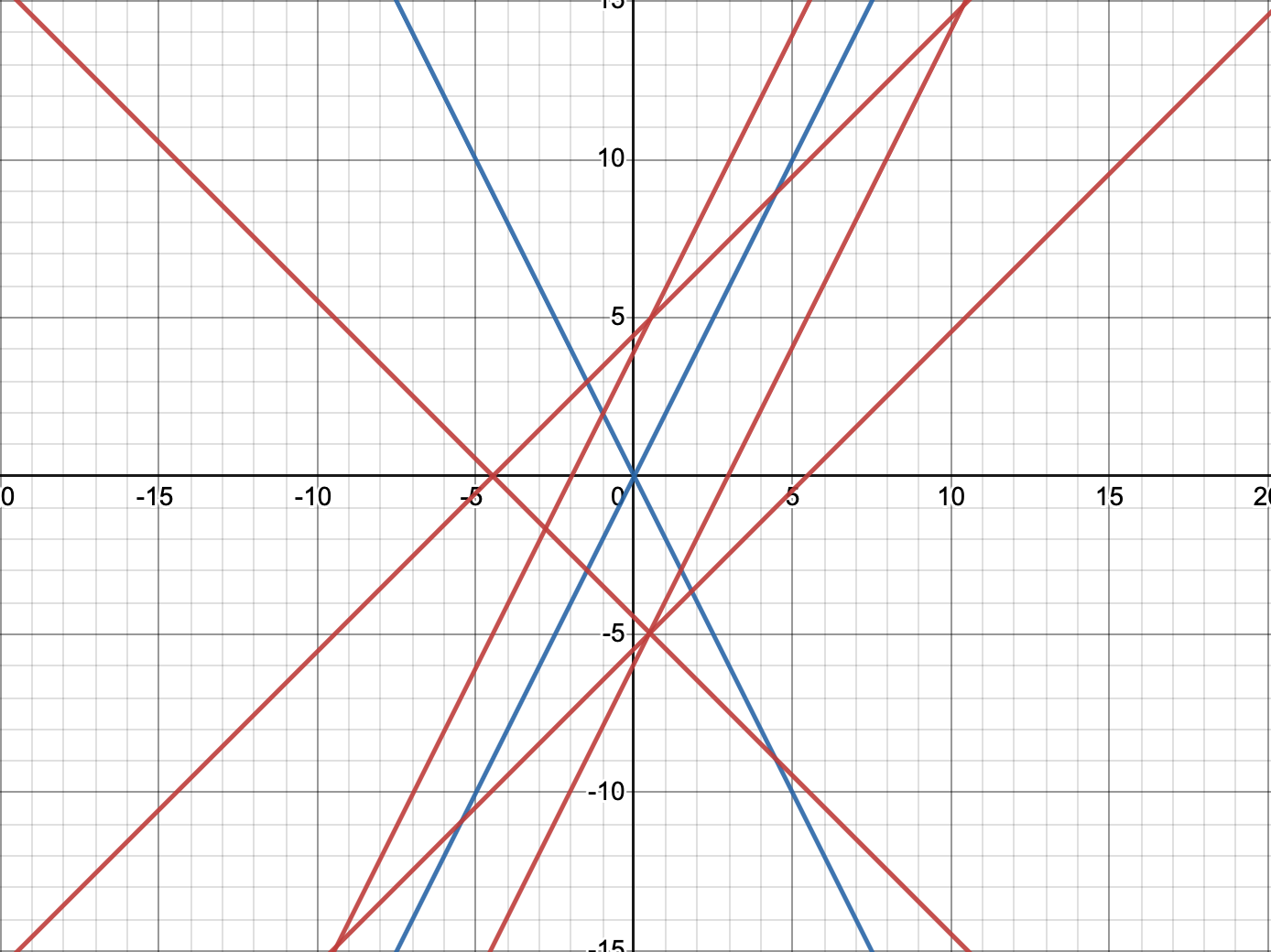}
\caption{\label{fig:graph44}Graph of the locus of $g(\left \| f(x)-2x\right \|,\left \| f(x)+2x\right \|) = 10$ in the vector space $U$}
\end{figure}

\begin{figure}
\centering
\includegraphics[width=0.75\textwidth]{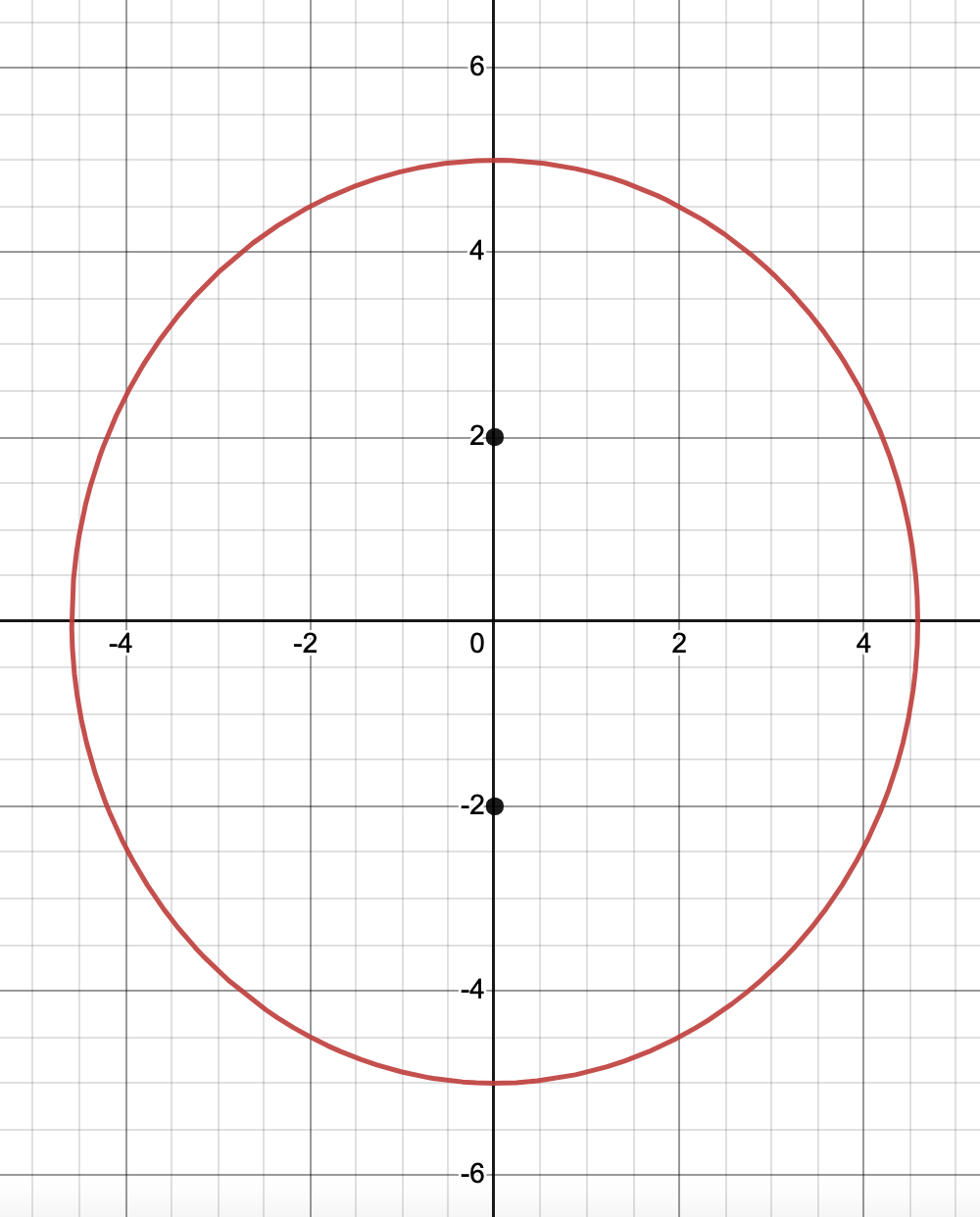}
\caption{\label{fig:graph45}Graph of the locus of $g(\left \| (x,y) - (0,2)\right \|, \left \| (x,y) - (0,-2)\right \|) = 10$ in the vector space $W$}
\end{figure}

\end{example}

As we can see in Figures \ref{fig:graph46} and \ref{fig:graph44}, the need for this isomorphism is clear. Evidently, trying to decipher any pattern amongst the locus by merely observing at these graphs can prove to be an impossible task. However, by applying the isomorphism to $\mathbb{R}^{n}$, one can more easily note the relationship between the locus and observe their behavior much more easily. This can be seen in Figure \ref{fig:graph47}, which includes a graph of the isomorphisms of all the locus in their respective vector spaces to $\mathbb{R}^{2}$. Note that the isomorphism of the locus in the vector space $V$ map to the same locus as the vector space $W$, that is, the locus of the isomorphism of $V$ map to the exact same locus as the same locus defined for $W$. This is due to the fact that the isomorphism of the real inner product of $V$ happens to be the same as the inner product defined in $\mathbb{R}^{2}$, however this does not always have to be the case. 
Now that we have this isomorphism in mind and that we know how to compute the relationship between vector spaces and $\mathbb{R}^{n}$, we can apply our theorems from Section \ref{sec:section3} to these loci in order to obtain some interesting results. Evidently, one can readily realize that we can apply Theorems \ref{thm:theorem2}, \ref{thm:theorem3}, \ref{thm:theorem4}. Note that, in order to apply these theorems, one must, one again, change the definition of a norm and set it to be the Euclidean definition of norm. This is because the definition of an angle between the vectors in a vector space does not always equate to the angle of the mapping of the vectors in $\mathbb{R}^{n}$. For that reason, we firstly have to change the definition of norm in $\mathbb{R}^{n}$ in order to find the mapping of the locus to $\mathbb{R}^{n}$, and then, in order to rotate the focus points, we have to make use of the definition $\left \langle a, b \right \rangle = a_{1}\cdot b_{1} + \dots + a_{n}\cdot b_{n}$. Once we find this linear transformation we can resort to using the alternate definition of norm which we obtained through the mapping of the inner product of $V$ to $\mathbb{R}^{n}$.

\subsection{Implications of Isomorphism on Geometry of $V$}

These last few proofs determining the equivalence through the isomorphism of the vector spaces $V$ and $\mathbb{R}^n$ provide a clear simplification, as expounded on above, on the visualization of the locus in the vector space $V$ to the more akin to our reality $\mathbb{R}^n$. However, the implications of the theorems above do not only apply solely to the visualization and calculation of locus in the vector space $V$, but also as a means to understand the geometry of the vector space $V$ itself. As such, we will devote this section, though unrelated to the rest of the paper, to the implications of the isomorphism between these vector spaces and the geometry of the vector space $V$

\subsubsection{Lines, Planes, and Bisectors}

Firstly, we will begin by exploring lines, planes, and bisectors in different vector spaces and how they map to their isomorphic $\mathbb{R}^n$. Perhaps we will find it easier to explore vector spaces isomorphic to $\mathbb{R}^3$ and $\mathbb{R}^2$, though this can be extended to higher dimensions. This not being the primary focus of the investigation, we will limit ourselves to make some basic observations of some vector spaces. 

Finding the perpendicular bisector of 2 points in $\mathbb{R}^2$, yielding a line, and in $\mathbb{R}^3$ is relatively easy, since it can be thought as the locus of points which are equally distanced from 2 focus points. Thus, the fact that points are equally distanced from both points implies that the difference of the distances is 0, meaning that in $\mathbb{R}^2$ and in $\mathbb{R}^3$, we would have the following equations respectively:

\[
\sqrt{(x-a_{0})^2+(y-a_{1})^2} - \sqrt{(x-b_{0})^2+(y-b_{1})^2}=0\]

\[
\sqrt{(x-c_{0})^2+(y-c_{1})^2+(z-c_{2})^2} - \sqrt{(x-c_{0})^2+(y-c_{1})^2+(z-c_{2})^2}=0\]

which we can write in our notation as:

\[g(\left\|(x,y)-(a_{1},a_{2})\right\|,\left\|(x,y)-(b_{1},b_{2})\right\|)=\]
\[\left\|(x,y)-(a_{1},a_{2})\right\| - \left\|(x,y)-(b_{1},b_{2})\right\|=0\]

\[g(\left\|(x,y,z)-(c_{1},c_{2},c_{3})\right\|,\left\|(x,y,z)-(d_{1},d_{2},d_{3})\right\|)\]

\[=\left\|(x,y,z)-(c_{1},c_{2},c_{3})\right\|-\left\|(x,y,z)-(d_{1},d_{2},d_{3})\right\|=0\]

\begin{example}
    
We can take, for example, the perpendicular bisector of points $(-1,3)$ and $(1,-2)$ using the equations above as seen in Figure \ref{fig:graph52}. Similarly, we can see the perpendicular bisector plane of the points $(-1,2,1)$ and $(2,-3,-2)$ as seen in Figure \ref{fig:graph53}.

\begin{figure}
\centering
\includegraphics[width=0.75\textwidth]{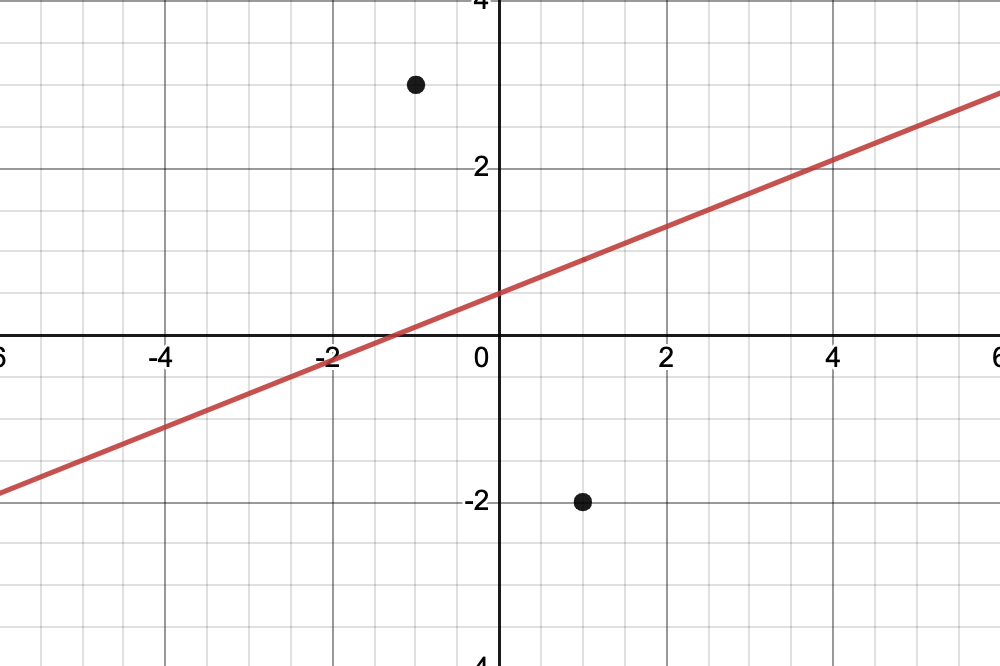}
\caption{\label{fig:graph52}Perpendicular bisector of points $(-1,3)$ and $(1,-2)$}
\end{figure}

\begin{figure}
\centering
\includegraphics[width=0.75\textwidth]{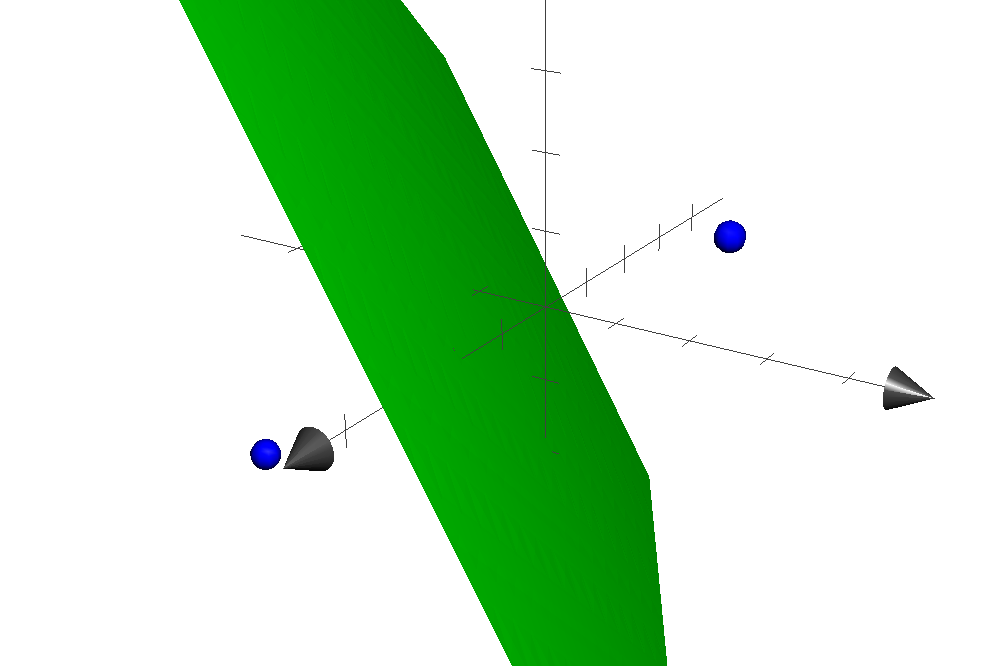}
\caption{\label{fig:graph53}Perpendicular bisector plane of points $(-1,2,1)$ and $(2,-3,-2)$}
\end{figure}

\end{example}

Having this in mind, we might want to move on to consider now a vector space isomorphic to either $\mathbb{R}^2$ or $\mathbb{R}^3$. We can simply look back and see the case we studied earlier with the vector space:

\begin{example}
    
\[
\left\{\begin{matrix}
U = P_{1}(R)\\ 
\left \langle f(x), g(x) \right \rangle_{U} = \int_{0}^{1} f(x) \cdot g(x)  dx \\

\end{matrix}\right.
\]

Similarly to the case we have in $\mathbb{R}^2$, we will obtain the perpendicular bisector of the polynomials $-1+3x $ and $1-2x$ because the coordinates of these matrices with respect to the standard ordered basis $\gamma$ of the vector space $V$, $\gamma = \{ 1, x\} $ are $(-1,3)$ and $(1,-2)$.

Lastly, we now need to find the equation that allows us to map the focus points of the equation in the vector space $U$ to $\mathbb{R}^2$. Again, using what we discussed above, we know that the mapping of the locus in this vector spaces to $\mathbb{R}^2$ is given by the relation:
\[
\left \langle f(x), g(x) \right \rangle_{U} = \left \langle a_{1} + a_{2}x, b_{1} + b_{2}x \right \rangle_{U} =
\left \langle (a_{1}, a_{2}), (b_{1}, b_{2}) \right \rangle =
\]
\[
\frac{6 a_{1}b_{1} + 3 a_{1}b_{2} + 3 a_{2}b_{1} + 2 a_{2}b_{2}}{6}
\]

Hence, the perpendicular bisector of these 2 polynomials can be found by solving the equation:

\[
\sqrt{\frac{6 (x+1)^{2} + 6 (x+1)(y-3) + 2 (y-3)^{2}}{6}} -
\]
\[
\sqrt{\frac{6 (x-1)^{2} + 6 (x-1)(y+2) + 2 (y+2)^2}{6}} = 0
\]

This equation is graphed in Figure \ref{fig:graph54}. Clearly, the mapping of the locus on $\mathbb{R}^2$ is very different form the one we observed using simply the vector space $\mathbb{R}^2$, indicating how the geometry of the vector space $U$ is considerably different form its isomorphic $\mathbb{R}^2$.

\begin{figure}
\centering
\includegraphics[width=0.75\textwidth]{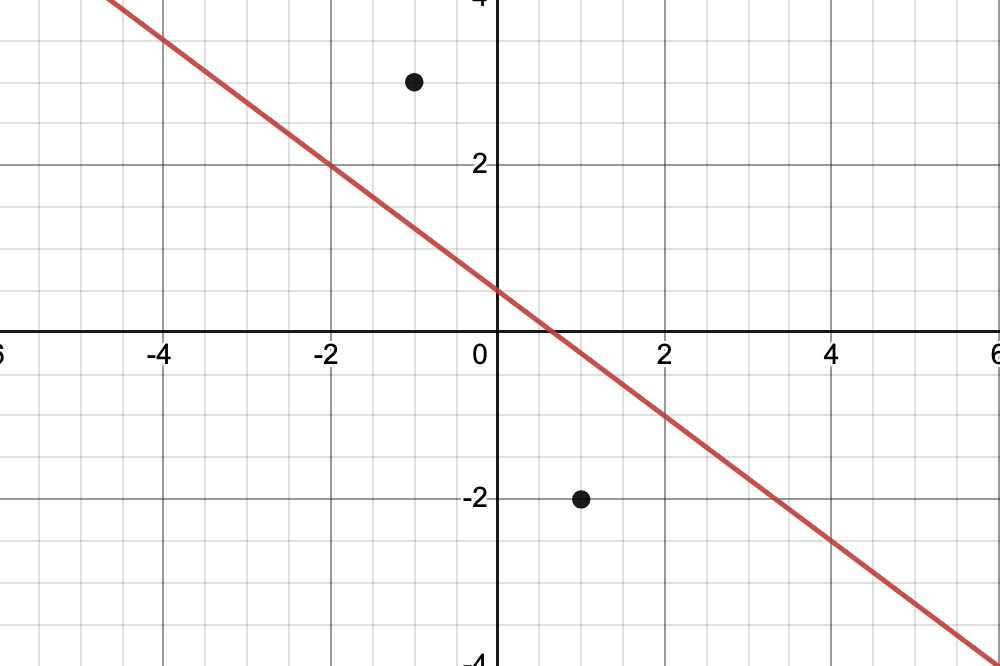}
\caption{\label{fig:graph54}Perpendicular bisector of polynomials $-1+2x$ and $1-3x$ mapped to $\mathbb{}R^2$}
\end{figure}

\end{example}

\subsubsection{Curves and Conics}

In the same spirit as before, we will now consider curves and conics in different vector spaces. Recalling the previous subsection, we can find that we have already explored the ellipse in different vector spaces and how it maps to $\mathbb{R}^2$. As such, refer to \ref{fig:graph47} to see how an ellipse with focus points $2x, -2x \in U$ maps to $\mathbb{R}^2$. 

A circle would be very much similar to en ellipse in that regard. For comparison, we will consider the circle centered around the polynomial $2x \in U$ with radius $\sqrt{10}$ and compare it to the ellipse. See Figure \ref{fig:graph55}, where the circle is the larger curve and the ellipse is the smaller.

\begin{figure}
\centering
\includegraphics[width=0.75\textwidth]{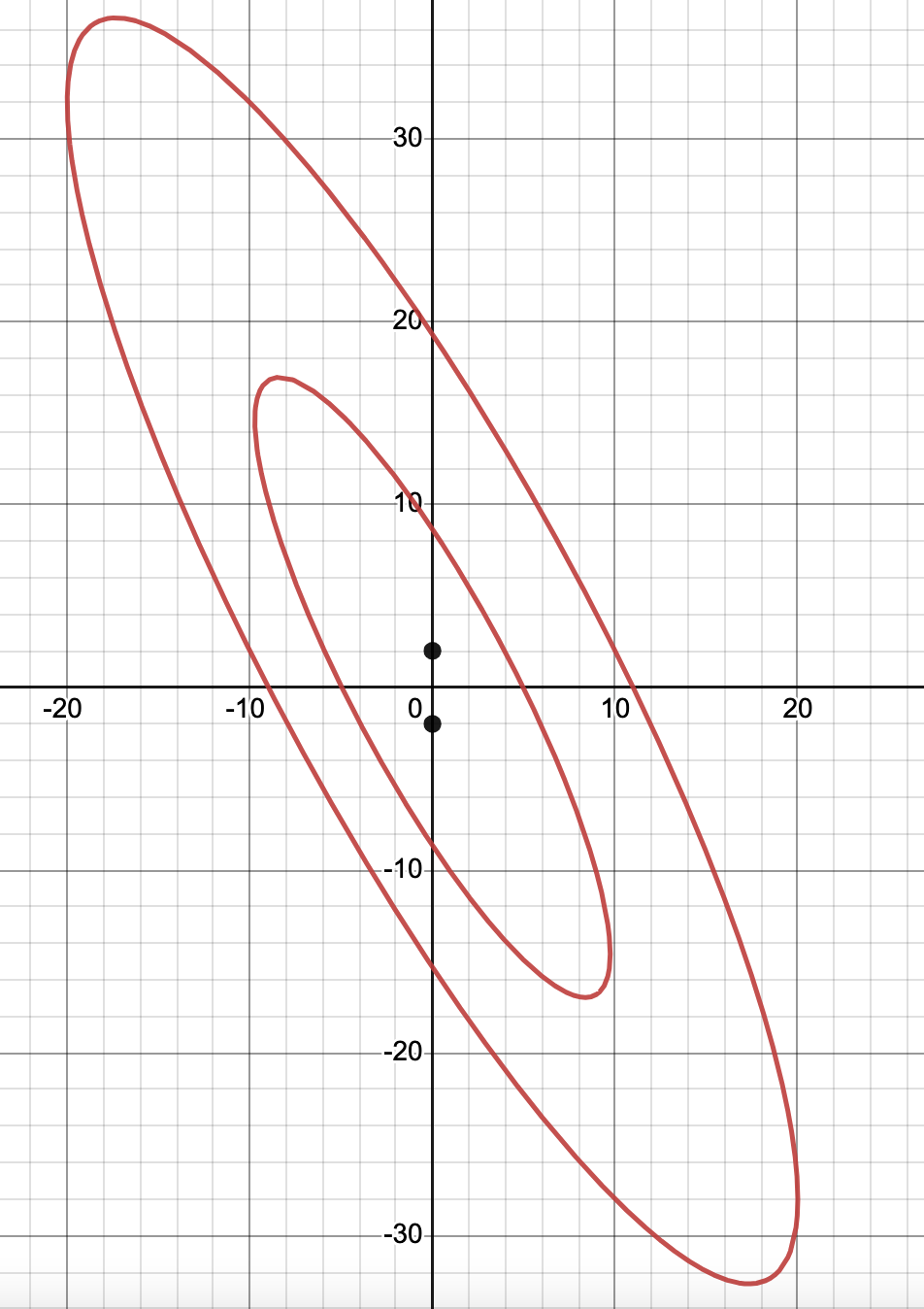}
\caption{\label{fig:graph55}Comparison of circle and ellipse in vector space $U$ mapped to $\mathbb{R}^2$}
\end{figure}

Finally, we can also consider hyperbolas in the vector space $U$ as a final example of how different the geometries in different vector spaces can be.

\begin{example}

Consider the hyperbola given by the equation:

\[
g(\left \| (x,y)-(0,2)\right \|,\left \| (x,y)-(0,-2)\right \|) = \left \| (x,y)-(0,2)\right \| - \left \| (x,y)-(0,-2)\right \| =
\]
\[
\sqrt{\frac{6x^2+6x(y-2)+2(y-2)^2}{6}} - \sqrt{\frac{6x^2+6x(y+2)+2(y+2)^2}{6}} = 1
\]

which we will compare to the hyperbola with focus points $(0,2), (0,-2) \in \mathbb{R}^2$ using the Euclidean definition of length. The graph of both these hyperbolas can be seen in Figure \ref{fig:graph56}

\begin{figure}
\centering
\includegraphics[width=0.75\textwidth]{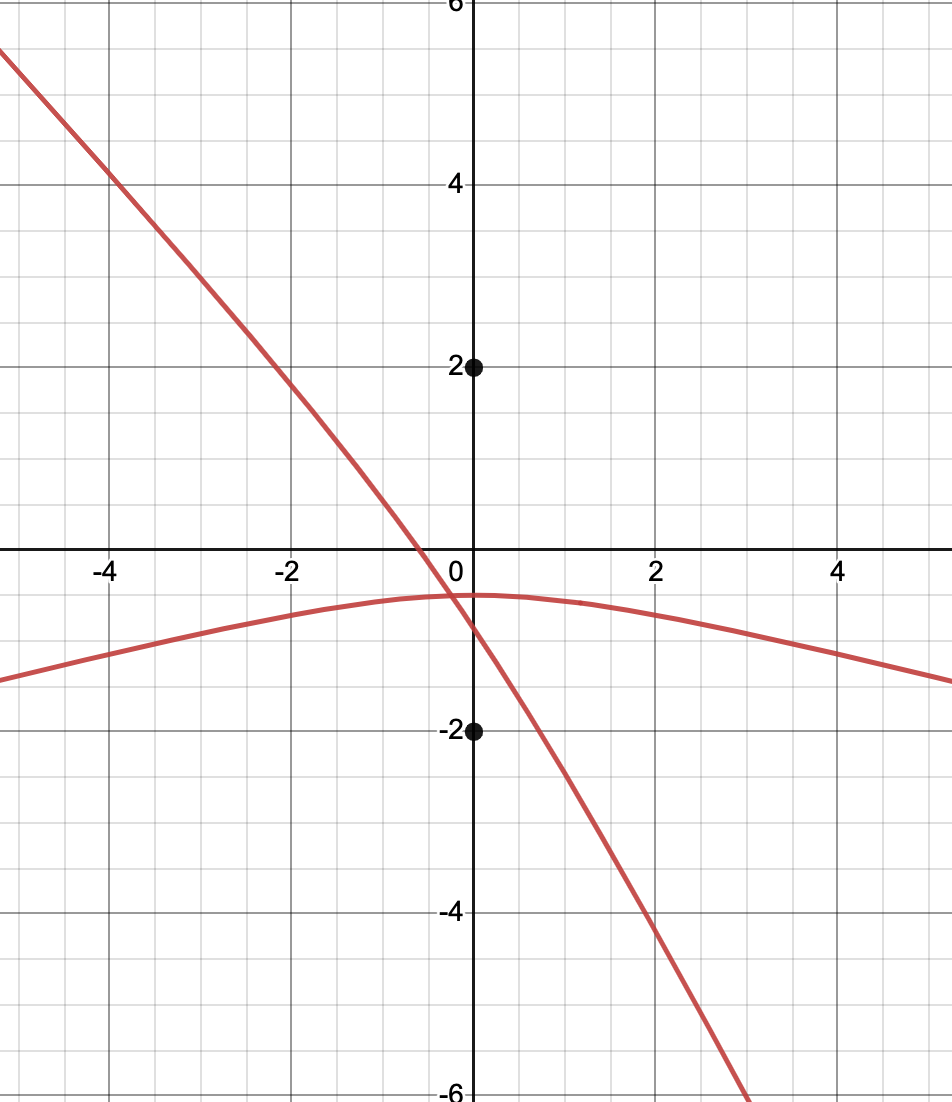}
\caption{\label{fig:graph56}Comparison of a hyperbola in vector space $U$ mapped to $\mathbb{R}^2$ and a hyperbola in $\mathbb{R}^2$}
\end{figure}

\end{example}

Evidently, these are but examples of what could be in the various different vector spaces and their multiple potential definitions of a norm. However, this does give insight to the fact that not all vector spaces share the same geometric properties and, thus, do not necessarily behave as one would expect in its isomorphic $\mathbb{R}^n$. 

\section{Acknowledgements}

I would like to express my gratitude for my Linear Algebra professor, Dr. Jone Lopez de Gamiz Zearra, who encouraged and supported me with this endeavor and pushed me to continue investigating and researching. I am thankful for her involvement in this paper. I would also like to thank my parents, who have given me their full and unconditional support for all these years. I also want to express my thankfulness towards Vanderbilt University by giving me the chance to conduct research with wonderful professors as well as obtaining a magnificent education.

\end{document}